\documentclass[12pt, a4paper]{amsart}

\pdfoutput=1

\usepackage[T1]{fontenc}
\usepackage[english]{babel}
\usepackage{fullpage}
\usepackage{amssymb}
\usepackage{mathptmx}
\usepackage{amsmath}    
\usepackage{amsfonts}  
\usepackage{soul} 

\usepackage{color}
\definecolor{Green}{rgb}{0,0.5,0}
\definecolor{Red}{rgb}{0.7,0,0}
\newcommand{\gt}[1]{\color{Green}{#1 }\color{black}}

\DeclareSymbolFont{newfont}{OML}{cmm}{m}{it}
\DeclareMathSymbol{\Epsilon}{3}{newfont}{15}
\DeclareMathSymbol{\ro}{3}{newfont}{37}

\DeclareMathOperator{\sing}{sing}
\DeclareMathOperator{\reg}{reg}
\DeclareMathOperator{\rank}{rank}

\newtheorem{proposition}{Proposition}[section]
\newtheorem{thm}[proposition]{Theorem}
\newtheorem{lemma}[proposition]{Lemma}
\newtheorem{corol}[proposition]{Corollary}
\theoremstyle{definition}

\newtheorem{example}[proposition]{Example}

\newtheorem{remark}[proposition]{Remark}
\newcommand{\CC}{\mathbb{C}}

\newcommand{\RR}{\mathbb{R}}
\newcommand{\DD}{\mathbb{D}}
\newcommand{\PP}{\mathbb{P}}

\newcommand{\be}{\begin{equation}}
\newcommand{\ee}{\end{equation}}
\newcommand{\V}{\mathcal{V}}
\newcommand{\A}{\mathcal{A}}
\newcommand{\B}{\mathcal{B}}
\newcommand{\rfo}{\ \ro\!\!(f^{-1}\!\!\!,f(\!\mathcal{V}\!)\!)}
\newcommand{\rfa}{\ \ro\!\!(f^{-1}\!\!\!,f(\!\mathcal{A}\!)\!)}
\newcommand{\rfb}{\ \ro\!\!(f^{-1}\!\!\!,\mathcal{B})}
\newcommand{\rf}{\ \ro\!\!(f\!,\!\mathcal{V})}
\newcommand{\rfA}{\ \ro\!\!(f\!,\mathcal{A})}
\newcommand{\lfA}{\mathcal{L}_\infty(f|\mathcal{A})}
\newcommand{\lfV}{\mathcal{L}_\infty(f|\mathcal{V})}
\usepackage[pdfstartview=FitH, pdfprintscaling=AppDefault, bookmarksopenlevel=4]{hyperref}
\usepackage{enumerate}
\begin{document}

\title[Green function on algebraic sets]{Łojasiewicz exponent and pluricomplex Green functions \\ on algebraic sets}
\author[L. Bia\l as-Cie\.{z}]{Leokadia Bialas-Ciez }
\address{Faculty of Mathematics and Computer Science, Jagiellonian University, {\L}ojasiewicza~6, 30-348 Krak\'ow, Poland}
\email{leokadia.bialas-ciez@uj.edu.pl}
\author[M. Klimek]{Maciej Klimek}
\address{Department of Mathematics, Uppsala University,
	P.O.Box 480, 751-06 Uppsala, Sweden}
\email{maciej.klimek@math.uu.se}

\date{\today}
\thanks{The work of the first author was partially supported by the National Science Centre, Poland, grant No. 2017/25/B/ST1/00906.}

\begin{abstract}  The results obtained in this article form a necessary foundation for any prospective development of an extension of the pluripotential theory to algebraic sets. We focus on the key concept of pluricomplex Green function on such sets, as well as its interplay with different growth characterizations for holomorphic mappings. 
In particular, if $f$ is a proper holomorphic mapping between two algebraic sets, then given a compact set $K$ in the range of $f$, we show how to estimate the pluricomplex Green functions of $K$ and of $f^{-1}(K)$ in terms of each other, the Łojasiewicz exponent of $f$ and the growth exponent of $f$. This result leads to explicit examples of pluricomplex Green functions on algebraic sets. We also present an enhanced version of the Bernstein-Walsh polynomial inequality specific to algebraic sets. This article provides a theoretical framework for future investigations of the rate of polynomial approximation of holomorphic functions on algebraic sets in the style of Bernstein-Walsh-Siciak theorem. 
\end{abstract}

\maketitle

\section{Introduction}

Given a non-polar compact set $K$ in the complex plane, one can define the Green function with pole at infinity associated with this set and harmonic in the unbounded complement of $K$. This concept has been widely used in univariate complex analysis for a very long time. The main areas of complex analysis in which the Green functions have played a crucial role are the approximation theory, constructions of conformal mappings and the classical potential theory (see e.g. \cite{Saf10} and \cite{Ran95}).     

\vskip 1mm

The pluricomplex Green functions are generalizations of the one dimensional notion to higher dimensions. Ever since the highly influential paper by Siciak \cite{Sic62} was published, the pluricomplex Green functions have continued to gain importance, particularly in the approximation theory and general study of plurisubharmonic functions. It is necessary to mention here the Siciak-Zaharjuta formula that gives a relation between the pluricomplex Green function and the extremal Siciak function defined by means of polynomials (see (\ref{Siciak-Za}) below). The formula leads to the Bernstein-Walsh polynomial inequality and to the well known Bernstein-Walsh-Siciak approximation theorem. Moreover, beginning with the fundamental work of Bedford and Taylor (see particularly \cite{BT76} and \cite{BT82}) in the late 1970s and early 1980s, the study of the pluricomplex Green functions has been linked to the non-linear potential theory built around the complex Monge-Ampère operator, and commonly referred to as the pluripotential theory (see e.g. \cite{Kl} and \cite{Kol05}). 

\vskip 1mm

Most of the research concerning the pluricomplex Green functions, and for that matter plurisubharmonic funcions in general, has been done in $\mathbb{C}^N$ or on complex manifolds. The papers by Sadullaev \cite{Sad83} from 1983 and Demailly \cite{Dem} from 1985 opened the door to the study of plurisubharmonic  functions on analytic sets. Over the years, this subject has been investigated by many authors (see e.g. \cite{BMT}, \cite{Col}, \cite{CGZ}, \cite{FN}, \cite{HM15}, \cite{Wik05}, \cite{Ze91}, \cite{Ze00}). 
If pluripotential theory is to be extended from $\mathbb{C}^N$ to algebraic sets, then the partial differential equations aspects of that theory become problematic, mainly because of singularities of algebraic sets. Instead, one has to focus of the approximation theory side of pluripotential theory. A fundamental framework serving this goal can be created by taking a close look at plurisubharmonic functions on analytic sets in general, by investigating properties of pluricomplex Green functions on algebraic sets and by analyzing different characterizations of growth of holomorphic mappings between algebraic sets, with particular attention paid to the Łojasiewicz exponents. Primarily, one has to scrutinize the interplay between all these notions, and this is the main objective of the results obtained in this article. 

\vskip 1mm
As a convenient vehicle for examining the interactions between all the concepts described in the previous paragraph, we will use a generalization of an invariance result for pluricomplex Green functions. 
In 1982 in \cite{Kl82a} and \cite{Kl82b} it was shown that, given a compact set in $\mathbb{C}^N$ and a proper polynomial mapping $f:\mathbb{C}^N\longrightarrow\mathbb{C}^N$, one can estimate the pluricomplex Green functions of $K$ and $f^{-1}(K)$ in terms of each other, and -- in certain cases -- even express one of these functions in terms of the other one (see Section 3 below for a precise statement). The result was new even in one variable. Only very special cases exist in older literature, like Fekete's considerations of the logarithmic capacity from 1930 (see \cite{Fek30})  or Walsh's construction of the Green function for some filled-in leminiscates from 1956 (see \cite{Wal56}). Also, in complex dynamics, there exists a well known formula for invariant attracting basins of infinity (see e.g. \cite[Thm. 6.5.3]{Ran95}).    

\vskip 1mm

The general case, in any dimension, apart from providing a plethora of explicit examples of pluricomplex Green functions and concrete polynomial estimates of Bernstein-Walsh type, has had far reaching consequences. They were mainly linked to complex dynamics, producing -- in particular -- completely new types of filled-in Julia sets, and to the theory of set-valued analytic functions (see e.g.
\cite{KlKo03}, \cite{KlKo06} and \cite{KlKo18}). Even in the one dimensional case, it yielded a new interpretation of the classical filled-in Julia sets as fixed points of set-valued contractions \cite{Kl95}. It can be expected that the main theorem of this paper will have similar consequences on algebraic sets to those of its version in $\CC^N$. 

\vskip 1mm

The paper is organized as follows. In the section following the introduction, we have gathered basic information concerning analytic and algebraic varieties. This paper is addressed mainly to complex analysts whose background in complex analytic geometry might be limited. Consequently, we felt that it would be prudent to include, also in other sections, some elementary details pertaining to that subject. In the subsequent section, we list briefly a few fundamental definitions of the pluripotential theory, followed by a precise statement of the theorem we intend to generalize. The somewhat longer 4th section deals with the notions of plurisubharmonicity and pluripolarity on analytic sets. The section ends with a technical lemma about the push-forwards of plurisubharmonic functions under proper holomorphic mappings between analytic sets. The ensuing 5th section is devoted to the pluricomplex Green functions on algebraic sets. In particular, we give a proof of a theorem about the relationship between the pluricomplex Green functions on algebraic sets and on the ambient vector space. This property is known, but it is sometimes incorrectly stated in some articles. Moreover, the proof given in \cite{Sad83}, where the theorem was first formulated, works only for algebraic curves but not for higher dimensional algebraic sets.
The penultimate 6th section examines the interrelationship between different descriptors of the growth of holomorphic functions on analytic sets. More specifically, given a holomorphic function $f$, we look at the growth and Łojasiewicz exponents of $f$, as well as the order of $f$ defined as the growth exponent of the multifuncion $f^{-1}$. Using the growth exponent we give an alternative description of the Siciak extremal function on algebraic sets as well as an enhanced version of the Bernstein-Walsh inequality.
The interplay between the growth indicators considered in Section 6 is crucial for the formulation and understanding of the main results of the paper, which are presented in the final section. The article closes with a few examples of applications of the main theorem. 

\section{Basic facts about analytic and algebraic sets}

Let $\Omega$ be a complex manifold of dimension $N$, e.g. a domain in $\mathbb{C}^N$. A set \ $\mathcal{V}$ \ is said to be \emph{analytic} in $\Omega$ if for every point $z\in \Omega$ there exist an open neighborhood  $U\subset \Omega$ of $z$, a positive integer $k$ and a holomorphic mapping $f=(f_1,...,f_k):U \rightarrow \CC^k$ such that $U\cap\mathcal{V}=\{w\in U \: : \: f_1(w)=...=f_k(w)=0\}$. If the condition holds only for each $z\in \mathcal{V}$, we say that the set $\mathcal{V}$ is {\it locally analytic}.
It follows from these definitions that $\V$ is analytic in $\Omega$ if and only if it is locally analytic and closed in $\Omega$.

\vskip 1mm
A point $z\in \mathcal{V}$ is \emph{regular} (\emph{of dimension $\ell$) of} $\mathcal{V}$ if $\V$ is a complex submanifold of $\Omega$ (of dimension $\ell$) in a neighborhood of $z$. Equivalently, the germ $\V_z$ can be written as $\{w \: : \:  g(w)=0\}_z$, where $g$ is such a holomorphic mapping from a neighborhood of $z$ to $\mathbb{C}^{N-\ell}$ that $g(z)=0$, and 
the differential $d_zg$ is surjective. The dimension of $\mathcal{V}$ at $z$ is denoted by dim$_z\mathcal{V}$. By $\reg\mathcal{V}$ we denote the collection of all regular points of $\mathcal{V}$. The {\it dimension of} $\mathcal{V}$ is defined as \ dim$\,\mathcal{V}:=\max\{{\rm dim}_z\mathcal{V} \: : \: z\in \reg\mathcal{V}\}$. The set of {\it singular points}  of $\mathcal{V}$ (denoted by $\sing\mathcal{V}$) is the set $\mathcal{V}\setminus\reg\mathcal{V}$. 
We say that $\mathcal{V}$ is of {\it pure dimension} $\ell$ \ if \ dim$_z\mathcal{V} =\ell$ \ for each point $z\in \reg\mathcal{V}$, i.e. the dimension of each 
component of $\mathcal{V}$ is $\ell$. Such sets are said to be {\it pure-dimensional} or {\it of constant dimension}. For a pure-dimensional analytic set $\mathcal{V}$, the set $\reg\mathcal{V}$ is a submanifold of $\Omega$, while $\sing\mathcal{V}$ is an analytic set in $\Omega$, which is closed and nowhere dense in $\mathcal{V}$. Moreover, $\dim\sing\mathcal{V}<\dim\mathcal{V}$.

\vskip 1mm

An analytic set $\mathcal{V}$ in $\mathbb{C}^N$ is {\it irreducible} if it is not the union of two analytic sets properly contained in $\mathcal{V}$. In this case, $\mathcal{V}$ is connected and pure-dimensional. 
It is known that an analytic set is irreducible precisely when its set of regular points is connected. 
Every connected locally irreducible analytic set is irreducible. The converse is usually not true. 

\vskip 1mm
If an analytic set in $\CC^N$ is compact, then it is finite and hence of dimension zero. Unless it is a singleton, it is also reducible.

\vskip 1mm

A subset $\mathcal{A}$ of $\CC^N$ is an {\it algebraic set} if $\mathcal{A}=\{z\in \CC^N\: : \: p_1(z)=...=p_k(z)=0\}$ for some polynomials $p_1,...,p_k\in \CC[z]$. 
Obviously, every algebraic set is analytic in $\CC^N$. An algebraic set $\A\subset\CC^N$ is {\it irreducible} if it is not the union of two algebraic sets properly contained in $\mathcal{A}$. This is the case if and only if $\A$ is irreducible as an analytic set, see e.g. \cite[VII.11.1]{Loj}.

\vskip 1mm

A linear space $X$ is called a \emph{Sadullaev subspace of $\CC^N$} for an analytic set $\mathcal{V}\subset\mathbb{C}^N$ if there exists a linear complement $Y$ of $X$ to $\CC^N$ such that
\begin{equation}\label{Sadullaev}
\mathcal{V} \subset \{ x+y\in \CC^N\: : \: x\in X, \ y\in Y, \ \|x\|\le C(1+\|y\|)\}
\end{equation} 
for some positive constant $C$. Unless otherwise specified, the symbol $\|\cdot\|$ will stand for the Euclidean norm, but in fact one can choose  any norm  in $\CC^N$ instead. If (\ref{Sadullaev}) holds, the natural projection $\pi: \mathcal{V}\ni x+y  \mapsto y\in Y$ is a proper mapping. Moreover, if dim$\, X= N- \ $dim$\mathcal{V}$, then the projection $\pi$ is surjective, see \cite[VII.7.1]{Loj}. According to theorems due to Rudin and Sadullaev (see e.g. \cite[VII.7.4]{Loj}), we have the following
\
\begin{thm} \label{Rudin_Sadullaev}
	An analytic set $\mathcal{V}\subset \CC^N$ of pure dimension $\ell$ is  algebraic if and only if there exists an $(N-\ell)$-dimensional Sadullaev space for $\mathcal{V}$.  Additionally, if $X$, $Y$ are linear spaces in $\CC^N$, $X\oplus Y=\CC^N$, $\V\subset X\times Y$ is a pure-dimensional analytic set and 
	\be \V \subset \{ (x,y)\in X\times Y\: : \: \|x\|\le C(1+\|y\|^c)\} \label{Sadullaev_coords}\ee
	for some $C,c>0$, then $\V$ is algebraic.
\end{thm}

To give an example, consider an algebraic hypersurface $\A=\{(x,y)\in\CC^N\: : \:f(x,y)=0\}$, where $f$ is the polynomial
\begin{equation*} 
f(x,y)=x^n+\sum\limits_{j=0}^{n-1} a_j(y)\: x^j, \ \ \ \ \ \ x\in \CC, \ \ y\in \CC^{N-1}, \ \ a_j\in \CC[y].
\end{equation*}
Then \ $\A\subset \{(x,y)\in \CC^N\: : \: |x|\le C(1+\|y\|^k)\}$, where \ $k:=\max_j \: {\rm deg}\, a_j$, \ $C=1+\sum\limits_{j=0}^{n-1}\|a_j\|$, 
\ $\|b\|:=\sum\limits_{|\alpha|\le \ell} |b_\alpha|$ \ for \ $b(y)=\sum\limits_{|\alpha|\le \ell} b_\alpha y^\alpha$,\ and
$\|y\|=\max_j|y_j|$ for $y=(y_1,\ldots,y_{N-1})\in\CC^N$.
 
\vskip 4mm

\section{Basic facts in pluripotential theory}

Let $D$ be an open set in $\CC^N$ and \ $u:D\rightarrow [-\infty, \infty)$ \ be an upper semicontinuous function on $D$, not identically $-\infty$ on any connected component of $D$. The function $u$ is {\it plurisubharmonic in $D$} and we write $u\in PSH(D)$, if for each $a\in D$ and $b\in \CC^N$, the function \ $t\mapsto u(a+bt)$ \ is subharmonic or identically equal to $-\infty$ on every connected component of the set $\{t\in\CC\,:\,a+bt\in D\}$. An equivalent definition says that $u\in PSH(D)$ if \ $u\circ h$ \ is subharmonic or identically $-\infty$ on the open unit disk $\DD$ in $\CC$ for every holomorphic function \ $h:\DD\rightarrow D$. \ Let $u^*$ denote the standard upper
regularization of $u$, i.e. \ $u^*(z)=\limsup\limits_{D\ni w\rightarrow z} u(w)$, \ $z\in D$.     

A plurisubharmonic function composed with a holomorphic mapping is plurisubharmonic. 
More precisely, for open sets $D_1$ and $D_2$ in $\CC^{N}$ and $\CC^{M}$, respectively, if \ $u\in PSH(D_2)$ \ and \ $\varphi:D_1 \rightarrow D_2$ \ is holomorphic, then \ $u\circ \varphi\in PSH(D_1)$.

\vskip 1mm

The concept of plurisubharmonicity extends in a natural way to complex manifolds via local charts. In particular, this is so in the case of complex submanifolds of $\CC^N$, irrespective whether they are closed or not. However, the situation becomes a little more complicated on analytic sets due to the presence of singularities. We will address this issue in the next section.

\vskip 1mm

Let $K$ be a subset of $\CC^N$. The set $K$ is said to be {\it pluripolar} if, for every $z\in K$, there exist an open neighborhood \ $U\subset \CC^N$ of $z$ 
and a function \ $u\in PSH(U)$ \ such that \ $u\equiv-\infty$ \ on $K\cap U$. The set $K$ is {\it non-pluripolar} if it is not pluripolar. For $N=1$ we say that $K$ is {\it polar} or, respectively, {\it non-polar}.

\vskip 1mm

The {\it pluricomplex Green's function} of a compact set $K\subset\CC^N$
is defined by 
\be \label{V_K} V_K(z)
:= \sup\{ u(z)\: : \: u\in \mathcal{L}(\CC^N) \ \ {\rm and} \ \ u\le 0
\ {\rm on} \ K \} \ \ \ \ {\rm for} \ \
z\in\mathbb{C}^N,\ee 
where $\mathcal{L}(\CC^N)$ is the
{\it Lelong class} of plurisubharmonic functions in $\mathbb{C}^N$
of logarithmic growth at the infinity, i.e.
\[\mathcal{L}(\CC^N)\!:=\!\{u \in PSH(\mathbb{C}^N) \: :
\: u(z)-\log(1+ \|z\|) \le c \ \ {\rm in} \ \ \CC^N \ \ {\rm for \ some} \ \ c=c(u) \}.\] 
By Siciak's theorem, either \ $V_K^*\in \mathcal{L}(\CC^N)$ \ 
or \ $V_K^*\equiv +\infty$. \ These properties 
are equivalent to non-pluripolarity  
or pluripolarity of $K$, respectively. For a non-polar set $K\subset \CC$, the function $V_K^*$ coincides with the Green's function
$g_{K}$ of the unbounded component of $\hat{\mathbb{C}}\setminus
K$ with logarithmic pole at infinity (as usual
$\hat{\mathbb{C}}=\mathbb{C}\cup\{\infty\}$).

The function $V_K$ is
closely related to polynomials in view of the following Siciak-Zaharjuta
formula: 
\be \label{Siciak-Za} V_K (z) \ = \ \log \Phi_K
(z) \ \ \ \ \mbox{for any compact set} \ \ K\subset\CC^N \ \ \mbox{and} \ \ z\in \mathbb{C}^N,\ee 
where $\Phi_K$ is
the {\it Siciak extremal function} defined by \[ \Phi_K(z) := \,
\sup\left\{|p(z)|^{1/{\rm deg}\, p}  : \: p\in
\CC[z]=\CC[z_1,...,z_N], \ {\rm deg}\, p\ge 1, \ \Vert p\Vert_K \le 1
\right\},\]
where $\Vert\ \cdot\ \Vert_K $ denotes the supremum norm.
\vskip 2mm

We refer to \cite{Kl} for the background information concerning this section.

\vskip 1mm
In the next sections we focus on the counterparts of the above notions and facts, but for analytic and algebraic sets
in place of $\CC^N$. 
The primary goal of the paper is to extend the following invariance theorem for pluricomplex Green functions in $\CC^N$ (see \cite{Kl82a}, \cite{Kl82b} or \cite[Thm. 5.3.1]{Kl}) to pluricomplex Green functions on algebraic sets.

\begin{thm} \label{klasyka}
	If $k,\ell$ are positive integers and \ $f=(f_1,...,f_N): \CC^{N}\rightarrow \CC^{N}$ \ is a holomorphic mapping, then the following conditions are equivalent:
	\vskip 2mm
	\begin{itemize}
		\item[(i)] $f$ is a polynomial mapping of degree at most $\ell$	and \ $\liminf\limits_{\|z\|\rightarrow \infty} \frac{f(z)}{\|z\|^k}>0$, \
		\item[(ii)] $f$ is a proper mapping and for every compact set $K\subset \CC^{N}$ 
			\[	k \: V_{f^{-1}(K)}\ \le \ V_K\circ f \ \le \ \ell \: V_{f^{-1}(K)} \ \ \ \ \ in \ \ \CC^{N}. \]
	\end{itemize}	
	If, in addition, $k=\ell$ then each of two above conditions is equivalent to the following one
	\begin{itemize}
		\item[(iii)]  $f$ is a polynomial mapping, {\rm deg}${\,f_1}=...={\rm deg}\,f_N=\ell$, and
		\be \label{homogeneous} \left(\hat{f_1},...,\hat{f_N}\right)^{-1}(0)=\{0\}, \ee
		where $\hat{f_j}$ denotes the homogeneous part of degree $\ell$ of the polynomial $f_j$. 
	\end{itemize} 	
\end{thm}

\vskip 4mm

\section{Plurisubharmonic functions on analytic sets}

  There are several non-equivalent definitions of plurisubharmonic functions on analytic sets. We will recall and discuss some of them. Let $\mathcal{V}$ be an analytic set of pure dimension $\ell>0$ in a complex manifold $\Omega$ of dimension $N$.  
  Consider a function $u: \mathcal{V} \rightarrow [-\infty,\infty)$ such that $u\not\equiv -\infty$ on every open subset of $\mathcal{V}$.  
  
\subsection{Upper semicontinuity on $\V$.} We begin with a convenient notational convention concerning upper semicontinuous regularizations of functions. If $X$ is a metric space, $A\subset X$ and $v:A\longrightarrow[-\infty,\infty)$ is a function, we define the \emph{upper semicontinuous regularization} of $v$ as the function
\[
v^{*A}(x):=\limsup\limits_{\substack{y\rightarrow x\\ \ y\in A}} v(y), \ \ \ \ x\in \overline{A}.
\] 

We will be primarily interested in two functions of this type:
\be \label{reg_na_reg} u^{*\reg\mathcal{V}}(z):=\limsup\limits_{\substack{w\rightarrow z\\ \ w\in \reg\mathcal{V}}} u(w), \ \ \ \ z\in \mathcal{V},
\ee 
\\[-10mm]
\be \label{reg_na_V} u^{*\mathcal{V}}(z):=\limsup\limits_{\substack{w\rightarrow z\\ \ w\in \mathcal{V}}} u(w), \ \ \ \ z\in \mathcal{V}, 
\ee
where $u$ is defined on $\reg\mathcal{V}$ or $\V$, respectively.
Clearly, $u^{*\reg\mathcal{V}} \le u^{*\mathcal{V}}$ and $u\le u^{*\mathcal{V}}$ on $\mathcal{V}$. On the other hand, $u^{*\reg\mathcal{V}} = u^{*\mathcal{V}}$ and $u\le u^{*\reg\mathcal{V}}$ on$\reg\mathcal{V}$.
The function $u$ is {\it upper semicontinuous on} $\mathcal{V}$ (or {\it on} $\reg\mathcal{V}$) if $u=u^{*\mathcal{V}}$ on $\mathcal{V}$ (or, respectively, $u=u^{*\reg\mathcal{V}}$ on $\reg\mathcal{V}$). This is consistent with the stsndard topological definition of upper semicontinuity: for all $c\in \mathbb{R}$ the set $\{u<c \}$ is open in $\mathcal{V}$ (resp. in $\reg\mathcal{V}$). Obviously upper semicontinuity on $\mathcal{V}$ is a stronger property than upper semicontinuity on $\reg\mathcal{V}$. 

Note that
\be \label{regularyzacje} u=u^{*\reg\mathcal{V}} \ \ {\rm on} \ \ \sing\mathcal{V} \ \ \ \ \Rightarrow \ \ \ \ u^{*\mathcal{V}}=u^{*\reg\mathcal{V}} \ \ {\rm on} \ \ \mathcal{V}.\ee 
Indeed, if it were not so, then for some $z_0\in \sing\mathcal{V}$ we would have $u^{*\reg\V}(z_0) < u^{*\mathcal{V}}(z_0)$. We could then  find a sequence $(z_n)$ such that $z_n\in \sing\V$, \ $\|z_n-z_0\|<\tfrac1{2^n}$ and 
\[ u(z_n)>\delta:=\tfrac{u^{*\mathcal{V}}(z_0) + u^{*\reg\V}(z_0)}2 > u^{*\reg\V}(z_0)\ \ \ {\rm for \ all} \ n.\]
Since $u(z_n)=u^{*\reg\V}(z_n)$, then for every $n$ there would exist $w_n\in \reg\V$ satisfying $\|w_n-z_n\| <\tfrac1{2^n}$ and $ u(w_n)>\delta$. This would mean that
\[ u^{*\reg\V}(z_0) \ge \limsup_{n\rightarrow \infty} u(w_n) \ge \delta,\] which is impossible since $\delta > u^{*\reg\V}(z_0)$. 

\vskip 1mm

We call the function $u$ {\it strongly upper semicontinuous} if  
\ $u=u^{*\reg\mathcal{V}}$ on $\mathcal{V}$. From (\ref{regularyzacje}), the strong upper semicontinuity of $u$ implies that $u=u^{*\mathcal{V}} \ {\rm on} \  \mathcal{V}$, i.e. $u$ is upper semicontinuous on $\mathcal{V}$. Obviously, the function $u^{*\mathcal{V}}$ is upper semicontinuous on $\mathcal{V}$ and $u^{*\reg\mathcal{V}}$ is strongly upper semicontinuous on $\V$.

\subsection{Plurisubharmonicity on $\V$.} \ 
Denote by $PSH(\reg\mathcal{V})$ the family of functions $u$ plurisubharmonic on $\reg\mathcal{V}$ in the sense of plurisubharmonicity on a complex manifold. Observe that for $u\!\!\in\!\! PSH(\!\reg\V\!)$ we have
\begin{description}
\item[(a)] $u$ is strongly upper semicontinuous on $\V$ if and only if \vskip -5mm
	\be \label{a} u=u^{*\reg\mathcal{V}} \ \ \ {\rm on \ \ } \sing\V; \ee
\item[(b)] $u$ is upper semicontinuous on $\V$ if and only if	
	\be \label{b} u=u^{*\mathcal{V}} \ \ \ \: {\rm \ \ \ on \ \ } \sing\V; \ee
\item[(c)] $u$ is locally upper bounded on $\V$ if and only if	
	\be \label{c}  u \ \ \textrm{is upper  bounded in a neighborhood of each point of} \sing\V.\ee
\end{description}
\noindent From (\ref{regularyzacje}), we have (\ref{a}) $\Rightarrow$ (\ref{b}) $\Rightarrow$ (\ref{c}). Of course, if the function $u$ is continuous in a neighborhood of $\sing\mathcal{V}$, then conditions (\ref{a}), (\ref{b}) and (\ref{c}) hold true.

\vskip 1mm

In  the affine case,  Lelong proved in 1945 that if $u$ is locally upper bounded on an open set $D\subset \CC^N$ and
	\be \label{subharm} \mbox{for every } \ a\in D, \ b\in \CC^N \ \ \mbox{the function} \ \ t\mapsto u(a+bt) \ \ \mbox{is subharmonic,} 
	\ee
then $u$ is upper semicontinuous in $D$. Lelong's question about the implication 
\[\mbox{(\ref{subharm})} \ \Rightarrow \ \mbox{upper semicontinuity of } u \mbox{ in } D \]
(without any assumption on $u$) seems to be still an open problem, see e.g. \cite[Remark 2.1.32]{JP}.  

On analytic sets with singular points, if a function $u\in PSH(\reg\,\V)$ is locally bounded on $\V$, it need not satisfy condition (\ref{a}) or condition (\ref{b}). Therefore, it is worth considering all three cases above. 

We say that $u$ is \emph{weakly plurisubharmonic on $\mathcal{V}$}, and we write $u\in psh(\mathcal{V})$, 
if $u\in PSH(\reg\mathcal{V})$ and $u$ is strongly upper semicontinuous. Such notion of plurisubharmonicity on analytic sets was studied by Braun, Meise and Taylor, e.g. in \cite{BMT}, and by Wikstr\"om in \cite{Wik05}. 
Plurisubharmonic functions  on $\reg\,\V$ satisfying condition (\ref{b}) (but not necessarily (\ref{a})) were considered by Hart and Ma'u in \cite{HM15} and and those satisfying (\ref{c}) by Sadullaev in \cite{Sad83}, Demailly in \cite[Def. 1.9, Thm. 1.10]{Dem} and Zeriahi in \cite{Ze00}, \cite[Def.1.3]{Ze91}.  Coman, Guedj, Zeriahi in \cite{CGZ}, \cite{Ze91} and Demailly in \cite[Def. 1.5, Thm. 1.6]{Dem} studied plurisubharmonic functions defined below by condition (\ref{psh_loc}).  
 
\vskip 1mm

From now on, we assume that $\V$ is an analytic set of pure dimension $\ell>0$ in $\Omega=\CC^N$, $N\ge 2$.
\vskip -4mm 

\begin{thm}\label{extension} Let \ $u: \V \rightarrow [-\infty, \infty) $ \ be a function such that $u\not\equiv -\infty$ on every open subset of $\mathcal{V}$. Then the following properties are equivalent: 
\be \label{psh_loc} {\it for \ every \ point} \ z\in \V \ {\it there \ exist \ an \ open \ neighborhood} \ U \ {\it of} \ z \ {\it in} \ \CC^N \ {\it and \ a \ function} \ee \[ v\in PSH(U) \ {\it such \ that} \ \ u_{|_{U\cap\V}}=v_{|_{U\cap\V}},\] 
\\[-10mm] 
\be \label{psh_disc}  u \ {\it is \ upper \ semicontinuous \ on \ } \mathcal{V} \ {\it and \ the \ function} \ \ u\circ h \ \ is \ subharmonic \ (or \ \equiv -\infty ) \ee \[ on \ the \ open \ unit \ disk \ \mathbb{D} \ {\it in} \ \CC \  {\it for \ every \ holomorphic \ function} \ h: \mathbb{D}\rightarrow \V, \] 
\\[-10mm]
\be \label{psh_glob} {\it the \ function} \ u \ {\it is \ the \ restriction \ to } \ \V \ {\it of \ a \ function} \ \tilde{u} \ {\it plurisubharmonic \ in} \ \CC^N.\ee 
\end{thm}

This is a consequence of two deep results:
\vskip 2mm
the Coltoiu theorem \cite[Prop.2]{Col}: \ (\ref{psh_loc}) $\Rightarrow $ (\ref{psh_glob}),
\vskip 2mm
the Forn\ae ss-Narasimhan theorem \cite{FN}: \ (\ref{psh_loc}) $\Leftrightarrow $ (\ref{psh_disc}).   

\vskip 2mm

A function $u$ is called {\it plurisubharmonic on $\mathcal{V}$}, and we write $u\in PSH(\V)$, if the above equivalent conditions hold. In other words, 
\[ PSH(\V)= PSH(\CC^N)_{|_{\V}}. \] 

\vskip 1mm

Note that if \ $u\in PSH(\V)$, then \ $u^{*\reg\V}\in psh(\V)$. \ Moreover, if $\V$ does not have singular points, then $psh(\V)=PSH(\V).$ On the other hand, a function $u$ which is weakly plurisubharmonic on $\V$  
need not be plurisubharmonic, as illustrated by the following simple example (see e.g. \cite{Ze91}).
\begin{example} \label{ex.easy}
	Let $\V=\{(w,z)\in \CC^2\: : \: wz=0\}$. Consider the function $u(w,0)=0$ for $w\in\CC\setminus\{ 0\}$ and $u(0,z)=1$ for $z\in \CC$. We can see that $u\in psh(\V)$ but, for the holomorphic map $h(w)=(w,0)$, by the maximum principle for subharmonic functions, condition (\ref{psh_disc}) is not satisfied and $u\not\in PSH(\V)$.  
\end{example} 

In the above example, the essential property of the analytic set $\mathcal{V}$ is its reducibility. The situation is quite different for locally irreducible sets. Demailly proved the following result, see \cite[Thm.1.7]{Dem}. 
 
\vskip -4mm

\begin{thm} \label{Dem_cor}
	Let $\mathcal{V}$ be a locally irreducible, pure-dimensional analytic set in $\mathbb{C}^N$. If \linebreak 
	$u\in PSH(\reg\V)$ and $u$ is locally bounded from above on the whole set $\V$, 
	then \ $u^{*\reg\mathcal{V}}\in PSH(\mathcal{V})$.  
\end{thm}

Recall that  every connected locally irreducible analytic set is irreducible, however the converse does not hold in general.

\vskip -3mm

\begin{corol} \label{psh=PSH}
	If $\V$ is a locally irreducible, pure-dimensional analytic set in $\CC^N$, then \[psh(\V) \subset PSH(\V).\]  	
\end{corol}

\vskip 2mm

\subsection{Pluripolarity on $\V$} 
The next theorem will serve also as a definition of locally and globally pluripolar subsets of analytic sets.

\begin{thm} \label{pluripolar} Let $K$ be a non-empty subset of an irreducible analytic subset $\V$ of a domain in $\CC^N$. The following properties are equivalent: 
	\be \label{pluripolar1} \: K\cap\reg \V \textit{ is \textbf{locally pluripolar} on  the  complex  submanifold } \reg\V,\ee 
\begin{quote}
i.e. for every $z\in K\cap\reg\V$ there  exists  an  open neighborhood  $U\subset \CC^N\setminus\sing\V$ of  $z$ and $u\in PSH(U\cap\V)$\footnote{The set $U\cap \V\subset \reg\V$ is a complex submanifold of $U$.} 
	 such  that $u\not\equiv -\infty$ and  $u\equiv-\infty$ on  $K\cap U$, 
\end{quote}	 
	\be \label{pluripolar2} \: K \textit{ is  \textbf{locally  pluripolar}  on } V\!\!,  \ee
\begin{quote}
	i.e. for every $z\in K$ there  exists  an  open  neighborhood $U\subset\CC^N$ of $z$
	and $v\in PSH(U)$ such  that $v\not\equiv -\infty$ on $U\cap \V$ and  $v\equiv-\infty$ on $K\cap U$,
\end{quote}
	\be \label{pluripolar3} \: K \textit{ is \textbf{globally pluripolar} on } \V\!\!, \ee \\[-10mm]
\begin{quote}
	i.e.  there  exists  a  function  in  the  Lelong  class $u\!\in\!\mathcal{L}(\CC^N)$ such  that $u_{|_\V}\not\equiv -\infty$ and $u\equiv-\infty$ on $K$.
	\end{quote}
\end{thm}

\begin{proof} 
	We can easily see that (\ref{pluripolar3}) $\Rightarrow$ (\ref{pluripolar2})  $\Rightarrow$ (\ref{pluripolar1}). Using the well-known Josefson's theorem about equivalence of local and global pluripolarity in $\CC^N$, Sadullaev showed  that (\ref{pluripolar1}) implies (\ref{pluripolar3}) (see \cite[Prop.1.6]{Sad83}). 
\end{proof}

In what follows, a non-empty subset $K$ of $\V$ will be called \emph{pluripolar on} $\V$, if it satisfies the equivalent conditions of Theorem~\ref{pluripolar}. In particular, if $\V$ is irreducible, then $\sing\V$ is pluripolar on~$\V$.

\vskip 1mm

The following property will be useful further on.
\begin{corol} \label{pluripolar_connectedness}
Let $\V$ be an irreducible analytic subset of a domain in $\CC^N$.
If $K$ is a closed pluripolar subset of $\V$, then the submanifold 
$\reg\V\setminus K$ is connected.
\end{corol}

\begin{proof} Without loss of generality, we may assume that $K\setminus\sing\V\neq\emptyset$.
For some plurisubharmonic function $u:\V\longrightarrow[-\infty,\infty)$
we have the inclusion $K\subset\{u=-\infty\}$.
Suppose that the conclusion of the corollary is false. Then $\reg\V\setminus K=A\cup B$, for some non-empty, disjoint, open subsets $A$ and $B$ of $\reg\V\setminus K$. Define a new function $v$ as being equal to $u$ on $A$ and to $-\infty$ on $K\cup B$.
Then $v$ is upper semicontinuous and fulfils condition (\ref{psh_disc}) of Theorem \ref{extension}. Hence, $v$ is plurisubharmonic on $\V$.
But then $v$ cannot be identically equal to $-\infty$ on $B$, which is a contradiction with the assumption that $B$ is non-empty.
\end{proof} 

\subsection{Plurisubharmonicity and  proper holomorphic mappings on $\V$} \label{psh&holo} 
Let $\V_1$ be an analytic set in an open subset of $\CC^{N}$ and $\V_2$ be an analytic set in an open subset of $\CC^{M}$.  

A continuous mapping $f: \V_1 \rightarrow \V_2$ is called {\it proper} if the pre-image of every compact set in $\V_2$ is compact. It is not difficult to show that a continuous mapping $f: \V_1 \rightarrow \V_2$ is proper if and only if it is a closed mapping with compact fibres (see e.g. \cite[B.2.4]{Loj}). 
 
A mapping \ $\varphi: \V_1 \rightarrow \V_2$ \ is \emph{holomorphic on $\V_1$} (and we write $\varphi\in\mathcal{O}(\V_1,\V_2)$) if every point in $\V_1$ has an open neighborhood $U\subset \CC^{N}$ such that $\varphi_{|_{U\cap\V_1}}=\tilde{\varphi}_{|_{U\cap\V_1}}$ for some holomorphic function $\tilde{\varphi}:U\rightarrow \CC^{M}$. 
Any holomorphic and proper mapping \ $\varphi: \V_1 \rightarrow \V_2$ \ is {\it finite}, i.e., for every $w\in \varphi(\V_1)$ the fibre $\varphi^{-1}(w)$ is a finite set of points. Indeed, since $\varphi$ is proper, the set $\varphi^{-1}(w)=\{z\in \V_1\: : \: \varphi(z)-w=0\}$ is a compact analytic subset of $\V_1$, and hence of $\mathbb{C}^N$. Consequently, $\varphi^{-1}(w)$ is a finite set (see e.g. \cite[Thm. 14.3.1]{Rud}).  

The following two fundamental results are both due to Remmert (see e.g. \cite[V.5.1, V.3.2]{Loj}) and \cite[V.6]{Loj}, respectively):

\begin{thm}\label{proper-map-thm} \textbf{(Proper Mapping Theorem)}
If $f: \V_1 \longrightarrow \V_2$ is proper and holomorphic on $\V_1$, then $f(\V_1)$ is an analytic subset of $\V_2$. Moreover, $\dim f(\V_1)=\dim\V_1$ and if, in addition, $\V_1$ is pure-dimensional, then so is $f(\V_1)$.
\end{thm} 
    
Observe that $f(\V_1)$ does not have to be an analytic set if we remove the assumption that $f$ is proper. Indeed, for the polynomial mapping $f(w,z)=(w^2,wz)$ and $\V_1=\CC^2$, the set $f(\V_1)$ \linebreak $=\CC^2\setminus\{(0,z):z\in \CC\setminus\{0\}\}$ is not analytic. Moreover, this example shows that the image of an algebraic set under a polynomial mapping is not necessarily an algebraic set.     

\begin{thm}\label{open-map-thm} \textbf{(Open Mapping Theorem)}
Let $\V_1,\V_2$ be pure-dimensional analytic sets of the same dimension and let $f:\V_1\longrightarrow\V_2$ be a holomorphic mapping. If $f$ has discrete fibres and $\V_2$ is locally irreducible, then $f$ is an open mapping (in the induced topologies on $\V_1$ and $\V_2$).
\end{thm}    

In particular, we can draw the following conclusion.

\begin{corol} \label{wymiar_obrazu}
Under the assumptions of the open mapping theorem, if $\V_2$ is connected and $f: \V_1 \rightarrow \V_2$ \ is proper, then $f$ is a surjection.
\end{corol}

\begin{remark}
	Of course a projection of algebraic set of dimension $k$ on a $k$-dimensional coordinate subspace does not have to be proper. For instance, if we take \ $\A=\{(z_1,z_2)\in \CC^2\: : \: z_1^2\,z_2-z_1=0\}$ \ and \ $\pi: \A\ni(z_1,z_2)\mapsto z_1\in \CC$, then \ $\pi^{-1}(\{z_1\: : \: |z_1-1|\le 1\})$ \ is not compact. On the other hand, if $\V$, $X$ and $Y$ are like in (\ref{Sadullaev_coords}) in Theorem \ref{Rudin_Sadullaev}, the projection $x+y\mapsto y$ is proper.
\end{remark}   

As in the classical case, we expect plurisubharmonicity to be invariant with respect to holomorphic substitutions. The next two propositions make this statement more precise.

\vskip 1mm

\begin{proposition} \label{a,b}
	Let $\V_1$, $\V_2$ be pure-dimensional analytic sets and  let $\varphi: \mathcal{V}_1\rightarrow \mathcal{V}_2$ be a holomorphic mapping. If $u\in PSH(\V_2)$, then $u\circ \varphi\in PSH(\V_1)$.
\end{proposition}

This statement is a simple consequence of the definition.

\vskip 1mm

Consider a holomorphic mapping \ $\varphi: \mathcal{V}_1\rightarrow \V_2$ \ between two pure-dimensional analytic sets $\V_1$ and $\V_2$. Recall that $\varphi$ is said to be {\it biholomorphic}, if it is a bijection and $\varphi, \varphi^{-1}$ are holomorphic. It is known that in this case  
\be \label {biholo} \varphi(\reg \V_1) = \reg\V_2, \ \ \ \mbox{and} \ \ \ \varphi(\sing\V_1) = \sing\:\V_2, \ee
see e.g. \cite[V.3.4]{Loj}. Moreover, $\dim\V_1=\dim\V_2$ and  $\varphi: \mathcal{V}_1\rightarrow \mathcal{V}_2$ is a homeomorphism (with respect to the induced topologies on $\V_1$ and $\V_2$). As a direct consequence, we get the following property.
	
\begin{proposition}
	Let $\V_1, \V_2$ be pure-dimesional analytic sets and \ $\varphi: \mathcal{V}_1\rightarrow \mathcal{V}_2$ \ be a biholomorphic mapping. Then  
	\[ u\in psh(\V_2) \ \ \Longleftrightarrow \ \ u\circ \varphi \in psh(\V_1).\] 
\end{proposition}	

\subsection{Basic lemma for proper holomorphic mappings} 

The following property of proper holomorphic mappings is crucial to our considerations.

\begin{lemma} \label{basic_lemma} Let $\mathcal{V}_1\subset \CC^{N}$ and $\mathcal{V}_2\subset \CC^{M}$ be irreducible analytic sets of the same dimension. 
If  $f: \mathcal{V}_1\rightarrow \mathcal{V}_2 $ \ is a proper holomorphic mapping on $\V_1$ and \ $u\in PSH(\V_1)$, then $f$ is surjective and the function \[v(w):=\max \ u(f^{-1}(w))=\max \{u(z)\: : \: z\in f^{-1}(w)\}, \qquad w\in \mathcal{V}_2\] is weakly plurisubharmonic, i.e. $v\in psh(\V_2)$. If in addition $\mathcal{V}_2$ is locally irreducible, then $v$ is plurisubharmonic, i.e. $v\in PSH(\V_2)$. \end{lemma} 
 
Sometimes such a function $v$ is referred to as the \emph{push-forward of $u$ under $f$}. A convenient notation is $f[u]:=v$.
 
\begin{proof} 
The fibres of $f$ are compact analytic sets in $\CC^N$, and so they are finite (see e.g. \cite[Thm.14.3.1]{Rud}). 
Since the mapping $f$ is both closed (being proper) and open (due to Theorem \ref{open-map-thm}), it is surjective (see Corollary \ref{wymiar_obrazu}). Consequently $v$ is well-defined. Furthermore, due to Theorem \ref{extension}, we may assume without loss of generality that $u\in PSH(\CC^N)$. Now, since any plurisubharmonic function on $\CC^N$ can be approximated by a decreasing sequence of continuous plurisubharmonic functions on $\CC^N$, we may also assume that 
$u$ itself is continuous. If this is the case, then also the function $v$ is continuous. Indeed, if $a,b\in\RR$ and $a<b$, we have
\[
v^{-1}\big((a,b)\big)=f\Big(u^{-1}\big((a,\infty)\big)\Big)
\setminus f\Big(u^{-1}\big([b,\infty)\big)\Big)
\]
and the set on the right-hand side is open. 

Note that if $W$ is a nowhere dense closed subset of $\V_2$, then $f^{-1}(W)$ is a nowhere dense closed subset of $\V_1$, and moreover
\[
f\big[u^{*(\V_1\setminus f^{-1}(W))}\big]\equiv f[u]\equiv\left(f[u]\right)^{*(\V_2\setminus W)}.
\]

Let 
\begin{align*}
S&:=f^{-1}\big(f(\sing\mathcal{V}_1)\cup \sing\V_2\big),\\
T&:=\{z\in\reg\mathcal{V}_1\setminus S\,:\,\rank_zf<\dim\V_2\},\\
Z&:=f\big(S\cup T\big)=f(\sing\mathcal{V}_1)\cup \sing\V_2\cup f(T).
\end{align*}
Now, because of Theorem \ref{proper-map-thm}, $f(\sing\mathcal{V}_1)$ is an analytic subset of $\V_2$ and hence of $\mathbb{C}^M$. So, $S$ is an analytic subset of $\V_1$ and hence of $\mathbb{C}^N$. Moreover, $S$ is nowhere dense in $\V_1$ since $f$ is open. Next, $\mathcal{V}_1\setminus S=\reg\mathcal{V}_1\setminus S$ is a closed submanifold of $\mathbb{C}^N\setminus S$ and $f$ restricted to $\reg\mathcal{V}_1\setminus S$ is a proper holomorphic mapping with values in $\mathcal{V}_2\setminus f(S)=\reg\mathcal{V}_2\setminus f(\sing\mathcal{V}_1)$, which is a closed submanifold of $\mathbb{C}^M\setminus f(S)$.  
By Theorem \ref{proper-map-thm}, $f(T)$ is an analytic subset of $\mathcal{V}_2\setminus f(S)$. Consequently, $Z$ is a closed nowhere dense subset of $\V_2$, which is the union of an analytic subset of $\V_2$ and an analytic subset of $\reg\mathcal{V}_2\setminus f(\sing\mathcal{V}_1)$. In particular, $Z$ is a pluripolar subset of $\V_2$. Also, $f^{-1}(Z)$ is nowhere dense in $\V_1$. At each point of $\V_1\setminus f^{-1}(Z)$ the mapping 
\begin{equation*}
g:=f\big|_{\V_1\setminus f^{-1}(Z)}\,:\,\V_1\setminus f^{-1}(Z)\longrightarrow\V_2\setminus Z
\end{equation*}
has maximal rank and hence is locally biholomorphic on the set $\V_1\setminus f^{-1}(Z)$, because of the Inverse Mapping Theorem. Furthermore, $g$ is proper. Consequently, all of the sets
\[
S_m:=\{w\in\V_2\setminus Z\,:\,\#f^{-1}(w)=m\},\qquad m=1,2,\ldots,
\]
are open in $\V_2\setminus Z$. This can be seen as follows. Take $m$ such that $S_m\neq\emptyset$. Let $w\in S_m$ and $g^{-1}(w)=\{z_1,\ldots,z_m\}$. Then, by the Inverse Mapping Theorem, there exist mutually disjoint neighborhoods $A_1,\ldots,A_m$ of $z_1,\dots,z_m$, respectively, and a neighborhood $B$ of $w$, such that $g\big|_{A_j}:A_j\longrightarrow B$ is a homeomorphism for each $j=1,\ldots,m$. We will be done if we can show that it is impossible to find a sequence $W=\{w_n\in B\,:\,n\in\mathbb{N}\}$ which is convergent to $w$ and such that 
$\#g^{-1}(w_n)>m$ for all $n\in\mathbb{N}$. Let us suppose that such a sequence exists. Then the compact set
$g^{-1}(W\cup\{w\})\setminus (A_1\cup\ldots A_m)$ would contain a convergent sequence $\{\xi_{i_n}\,:\,n\in\mathbb{N}\}$, such that $g(\xi_{i_n})=w_{i_n}$ for all $n$. But then $g(\lim_{n\to\infty}\xi_{i_n})=w$
and $\lim_{n\to\infty}\xi_{i_n}\not\in A_1\cup\ldots A_m$, which would contradict the fact that $w\in S_m$.

Since $Z$ is pluripolar in $\V_2$, the set $\V_2\setminus Z$ is connected by Corollary \ref{pluripolar_connectedness}. Consequently, only one of the sets $S_m$ is non-empty. In other words, there exists $n\in\mathbb{N}$ such that the mapping $g$ is an $n$-to-one local homeomorphism.
  
We have to show that $v$ is weakly plurisubharmonic. First, note that $v$ is plurisubharmonic on $\V_2\setminus Z$ because the mapping $g$ is locally biholomorphic. Secondly, the removable singularity theorem for plurisubharmonic functions \cite[Thm 2.9.22]{Kl} transfers easily to complex manifolds. Thus, since $v$ is continuous and $Z\cap\reg\V_2$ is pluripolar in $\reg\V_2$, the function $v$ is actually plurisubharmonic on $\reg\V_2$. This is enough to conclude that $v$ is weakly plurisubharmonic on $\V_2$, as required. Now, the last statement of the lemma is true in view of Demailly's Theorem \ref{Dem_cor}.
\end{proof}
\vskip 4mm

\section{Pluricomplex Green's function on analytic sets}

Let $\V$ be an analytic set of positive pure dimension in $\CC^N$.
\subsection{The Lelong classes on $\V$} We have three natural ways to define the three {\it Lelong classes on} $\V$:  
\begin{align*}
\mathcal{L}(\V)&:=\{u\in PSH(\V) \,: \,  u(z) - \log(1+\|z\|) < c \mbox{ \ on \ } \V \mbox{ \ for some \ } c=c(u)\},\\
\mathcal{L}'(\V)&:=\{u\in psh(\V) \, : \,  u(z) - \log(1+\|z\|) < c \mbox{ \ on \ } \V \mbox{ \ for some \ } c=c(u)\},\\
\mathcal{L}(\CC^N)_{|_\V}&:=\{u_{|_\V} \, : \, u\in PSH(\CC^N), \  u(z) - \log(1+\|z\|) < c \mbox{ \ on \ }  \CC^N \mbox{ \ for some \ } c=c(u)\},
\end{align*}
where $\|\cdot\|$ is a fixed norm in $\CC^N$, say the Euclidean one.

By Demailly's theorem, we have $\mathcal{L}'(\V)\subset \mathcal{L}(\V)$ for a locally irreducible set $\V$. Example \ref{ex.easy} shows that this inclusion is not true in general. 

\vskip 1mm

Obviously, we have the inclusion \ $\mathcal{L}(\CC^N)_{|_\V} \subset \mathcal{L}(\V)$.
A full characterization of functions in $\mathcal{L}(\A)$ that belong to $\mathcal{L}(\CC^N)_{|_\A}$ for an algebraic set $\A\subset\CC^N$ was discovered by Coman, Guedj and Zeriahi in \cite[Prop.3.1]{CGZ}. As a consequence, they obtained the following sufficient condition for the equality $\mathcal{L}(\CC^N)_{|_\A} = \mathcal{L}(\A)$.  

\begin{thm}
	If $\A$ is an algebraic set in $\CC^N$ and the germs $\overline{\A}_a$ are irreducible for all \linebreak $a\in \overline{\A}\setminus\A$, then \[\mathcal{L}(\CC^N)_{|_\A} = \mathcal{L}(\A),\] where $\overline{\A}$ is the closure of $\A$ in the complex projective space $\PP^N$. 
\end{thm}	

With the help of this result, a counterexample to the inclusion $\mathcal{L}(\V)\subset \mathcal{L}(\CC^N)_{|_\V}$ can be constructed.

\begin{example} (see \cite[Ex.3.3]{CGZ})
Consider the algebraic set $\A=\{ (w,z)\in \CC^2 \ : \ wz=1+w^3 \}$ and the function $u(w,z)=\max\{ -\log |w|, 1+2\log |w| \}$. We see that $u\in PSH((\CC\setminus \{0\})\times \CC)$ and \ $u_{|_\A}=\tilde{u}_{|_\A}$, where $\tilde{u}(w,z)=\max\{ \log|z-w^2|, 1+2\log |w| \}\in PSH(\CC^2)$. The function $u$ is of logarithmic growth on $\A$ and so $u\in \mathcal{L}(\A)$. On the other hand, 
$\tilde{u}\notin \mathcal{L}(\CC^2)$ and $u\not\in\mathcal{L}(\CC^2)_{|_\A}$.
\end{example}

Since the set $\A$ from the above example does not have singular points, we get $psh(\A)=PSH(\A)$, and so $\mathcal{L}(\A)=\mathcal{L}'(\A)$. Therefore, the example shows that also the inclusion \ $\mathcal{L}'(\V)\subset \mathcal{L}(\CC^N)_{|_\V}$ does not hold in general. 

\vskip 2mm

\subsection{Pluricomplex Green functions on $\V$}
Let $K$ be a subset of an analytic set $\mathcal{V}$ of positive pure dimension in $\CC^N$. The pluricomplex Green function $V_K$ is defined by (\ref{V_K}). By Siciak's theorem $V_K^*=+\infty$, because $K$ is pluripolar as a subset of an analytic set. However, $V_K$ can be locally bounded on $\V$. The fundamental result in this context is the following Sadullaev's theorem:

\begin{thm}\cite[Prop.2.1, Thm.2.2]{Sad83} \label{lok_bound_VK}
	Let $\V$ be an irreducible analytic set in $\CC^N$ and $K$ be a compact subset of $\V$. Then $V_K$ is locally bounded from above on $\V$ if and only if $\V$ is an algebraic set and $K$ is non-pluripolar on $\V$.
\end{thm}

In addition to the above function $V_K$, we will study here the {\it pluricomplex Green function} 
{\it of $K$ on an analytic set $\mathcal{V}$} that is defined as
\[ V_{K,\mathcal{V}}(z):=\sup\{u(z)\: : \: u\in \mathcal{L}(\mathcal{V}) \mbox{ \ and \ } u_{|_K}\le 0\}, \mbox{ \ \ \ \ } z\in \mathcal{V}.\] 
The compactness of $K$ is usually assumed, but  it is not necessary for the above definition and some basic properties of pluricomplex Green functions. 
 
\vskip 2mm 
 
One can ask about a relationship between $V_{K,\V}$ and the pluricomplex Green function $V_K$. Clearly, \ $V_{K,\CC^N}=V_K$. Moreover, $V_K$ and $V_{K,\V}$ are identical on $\V$, as we can see in the following

\begin{thm} \label{VKV=VK}
	Let $\mathcal{V}$ be an analytic set of positive pure dimension in $\CC^N$ and $K$ be a compact subset of $\V$. Then
	\[ V_{K,\V}(z) = V_K(z),\quad z\in \V. \]
    Consequently, $\exp(V_{K,\V})$ is equal to the Siciak extremal function $\Phi_K$ on the set $\V$.
	Furthermore, if $\V$ is irreducible, then $V_{K,\V}$ is locally bounded from above if and only if $\V$ is algebraic and $K$ is a non-pluripolar subset of $\V$. \end{thm}

The above theorem is well-known, but in some papers other versions (not always correct) are 
presented. It first appeared as
Proposition 3.4 in \cite{Sad83} asserting the equality of locally bounded pluricomplex Green functions on algebraic curves with their global counterparts. After that proposition the author claims that the same argument would work in higher dimensions depending on a property of maximal plurisubharmonic functions. Unfortunately, the statement is incorrect as in dimensions higher than one maximal plurisubharmonic functions are generally not plurisuperharmonic. Here we present a proof that works in any dimension and is based on the following result by Coman, Guedj and Zeriahi.

\begin{thm} \label{CGZ} {\rm (see \cite[Thm.A]{CGZ})} \ 
	Let \ $f\in PSH(\V)$, \ $g\in PSH(\CC^N)$. \ If \ $f<g$ on $\V$ \ and $g$ is continuous in $\CC^N$, then for all \ $b>1$ \ there exists \ $h=h_b\in PSH(\CC^N)$ \ satisfying two conditions: 
	\[ h_{ \ |{\V}}=f \ \ \mbox{ and } \ \ h(z)< b\: \max\{0,g(z)\} \ \ \mbox{for } \ z\in\CC^N. \]
\end{thm}	

\noindent {\it Proof of Theorem} \ref{VKV=VK}.
	Since \ $\mathcal{L}(\CC^N)_{|_\V} \subset \mathcal{L}(\V)$, \ we have \ $V_K\le V_{K,\mathcal{V}}$ on $\V$. We will show that $V_K(z)\ge V_{K,\mathcal{V}}(z)$ for $z\in\V$. Fix a function $u\in \mathcal{L}(\V)$ such that $u_{|_K}\le 0$. For some constant $c(u)$, we have \ $ u(z) - \log(1+\|z\|) < c(u) \mbox{ \ on \ } \V$. Set $f=u-c(u)$ and $g(z)=\log (1+||z||)$. If $\|\cdot\|$ is the Euclidean norm, then $g\in PSH(\CC^N)$ and is continuous in $\CC^N$. By Theorem \ref{CGZ}, for any $b>1$, we take a function $h=h_b\in PSH(\CC^N)$ such that \[ h_{|_\V} = u-c(u) \ \ \mbox{and} \ \ h_b(z)<b\: \max\{0,\log(1+\|z\|)\} = b\: \log(1+\|z\|), \ \ \ \ z\in \CC^N.\]  
	For a fixed $z\in \CC^N$, let 	
	\begin{multline*}
	V_b(z):=\sup\{ \tilde{u}(z)\: : \: \tilde{u}\in PSH(\CC^N), \ \tilde{u}(z)<c+ b\: \log(1+\|z\|) \ \mbox{for some $c$ and} \ z\in \CC^N, \ \tilde{u}_{|_K}\le 0\}\\
		\ge \sup\{ \tilde{u}(z)\: : \: \tilde{u}\in PSH(\CC^N), \ \tilde{u}(z)<c(u) + b\: \log(1+\|z\|) \ \mbox{for} \ z\in \CC^N, \ \tilde{u}_{|_K}\le 0\} \phantom{V}
\end{multline*}				
	and the function $c(u)+h_b(z)$ is an element of the family over which the last sup is taken. Therefore, for every $z\in \V$, we have	
	\[ V_b(z)\ge c(u)+h_b(z)=u(z). \]
	Since $u$ is an arbitrary function from $\mathcal{L}(\V)$ such that $u_{|_K}\le 0$, we get the inequality 
	\[ V_b\ge V_{K,\V} \ \ \ \ \mbox{on} \ \V \ \ \mbox{for any} \ \ b>1.\]
	On the other hand, 
	\begin{align*}
	\tfrac1b \: V_b(z) &=\sup\left\{ \tfrac{\tilde{u}(z)}b\: : \: \tilde{u}\in PSH(\CC^N), \ \tfrac{\tilde{u}(z)}b<\tfrac{c}b+ \log(1+\|z\|) \ \mbox{for some $c$ and} \ z\in \CC^N, \ \tilde{u}_{|_K}\le 0\right\}\\
	&= \sup\left\{ v(z)\: : \: v\in PSH(\CC^N), \ v(z)<d+ \log(1+\|z\|) \ \mbox{for some $d$ and} \ z\in \CC^N, \ v_{|_K}\le 0\right\}\\
	&=V_K(z), \ \ \ z\in \CC^N.
	\end{align*}
	Consequently, for any $z\in \V$,
	\[ V_K(z)= \tfrac1b\: V_b(z) \ge \tfrac1b \: V_{K,\V}(z). \]
	By letting $b$ tend to 1, the assertion follows. 
	
	The second conclusion of the theorem follows from Sadullaev's theorem above.\hfill $\Box$

\vskip 2mm

Consider also the pluricomplex Green function related to $\mathcal{L}'(\V)$
\[ U_{K,\mathcal{V}}(z):=\sup\{u(z)\: : \: u\in \mathcal{L}'(\mathcal{V}) \mbox{ \ and \ } u_{|_K}\le 0\}, \mbox{ \ \ \ \ } z\in \mathcal{V}.\]

\begin{thm}
	If $\V$ is a locally irreducible analytic set of positive pure dimension in $\CC^N$ and $K$ is a compact set in $\V$, then \[ U_{K,\V}^{ \ *\reg\V} = V_{K,\V}^{ \ *\reg\V} \ \ \ \ \mbox{and} \ \ \ \ U_{K,\V} \le V_{K,\V}  \ \ \ \ \mbox{on} \ \ \V. \]
\end{thm}

\begin{proof}
	Fix $z\in \reg\V$ and \ $u\in\mathcal{L}(\V)$. Since $u\in PSH(\reg\V)$, we have $u^{*\reg\V}(z)=u(z)$, \ $u^{*\reg\V}\in psh(\V)$ and $u^{*\reg\V}\in \mathcal{L}'(\V)$. \ Moreover, $\mathcal{L}(\V)\subset PSH(\V)=PSH(\CC^N)_{|_\V}$, and so $u=\tilde{u}_{|_\V}$ for some function $\tilde{u}\in PSH(\CC^N)$. Thus, we obtain \ $u^{*\reg\V}=\tilde{u}^{*\reg\V}\le \tilde{u}^{*\CC^N} = \tilde{u}$ \ on $\V$ and \ $u^{*\reg\V}_{ \ |_K}=\tilde{u}^{*\reg\V}_{ \ |_K}\le \tilde{u}_{|_K}=u_{|_K}\le 0$. \ Consequently, 
	\begin{align*}
	V_{K,\V}(z)&=\sup\{u(z)\: : \: u\in \mathcal{L}(\V), \ u_{|_K}\le 0\} = \sup\{u^{*\reg\V}(z)\: : \: u\in \mathcal{L}(\V), \ u_{|_K}\le 0\} \\
	&\le \sup\{v(z)\: : \: v\in \mathcal{L}'(\V), \ v_{|_K}\le 0\} = U_{K,\V}(z)
	\end{align*}
	and hence
	\[ V_{K,\V}^{ \ *\reg\V} \le U_{K,\V}^{ \ *\reg\V} \ \ \ \ \mbox{on} \ \ \V. \]
	
	On the other hand, by Corollary \ref{psh=PSH}, $psh(\V)\subset PSH(\V)$, hence $\mathcal{L}'(\V)\subset\mathcal{L}(\V)$ and $U_{K,\V} \le V_{K,\V}$ \ on $\V$.
\end{proof}

\begin{example} \label{nplp_on_A}
	Let \ $B(a,\rho):=\{z\in \CC^N \: : \: \|z-a\|<\rho\}$  and let $\bar{B}(a,\rho)$ denote the corresponding closed ball. Consider an irreducible algebraic set $\A\subset \CC^N$ and fix $a\in \A$. We will estimate $V_{\A\cap \bar{B}(a,\rho),\A}$ on $\A$. Assume that $a=0$ taking a translation of $\A$ if necessary.
	By Sadullaev's criterion for algebraicity (see Theorem \ref{Rudin_Sadullaev}),
	there exists a linear space $X\subset\CC^N$ and a linear complement $Y$ of $X$ to $\CC^N$ such that
	\[ \A \subset \{ z= z_x+z_y\in \CC^N\: : \: z_x\in X, \ z_y\in Y, \ \|z_x\|\le C(1+\|z_y\|)\}\] 
	for some positive constant $C$. 
	Set
	\begin{align*}
	B_r:&=\{z=z_x+z_y\in \A\: : \: z_x\in X, \ z_y\in Y, \ \|z_y\|\le r\}\\ &\subset \{ z= z_x+z_y\in \A \: : \: z_x\in X, \ z_y\in Y, \ \|z_y\|\le r, \ \|z_x\|\le C(1+r)\}. \end{align*}
	As in the proof of \cite[Thm.2.2]{Sad83}, we can show that 
	\[ V_{B_r}(z) = \log^+ \frac{\|z_y\|}{r} \quad\mbox{for }  \ z=z_x+z_y\in \A. \]
	Consider \ $r>0$ \ such that \ $B_r \subset \A\cap \bar{B}(0,\rho)$. \ 
	For any $z=z_x+z_y\in \A$, we have
	\begin{equation*}  
	\log^+ \frac{\|z\|}{\rho} =  V_{\bar{B}(0,\rho)}(z) \le V_{\A \cap \bar{B}(0,\rho)}(z) = 
	V_{\A \cap \bar{B}(0,\rho),\A}(z) \le V_{B_r}(z) = \log^+ \frac{\|z_y\|}{r} \le \log^+ \frac{\|z\|}{r}.
	\end{equation*}
	In particular, the set $\A \cap \bar{B}(0,\rho)$ is non-pluripolar on $\A$ (see Theorem \ref{lok_bound_VK}).
\end{example}

\vskip 4mm

\section{Growth exponent, order and Łojasiewicz exponent of a function}

Definitions of the notions mentioned in the title of this section can be stated for any unbounded set $\V\subset\CC^N$ and any (unbounded) function $f$ defined on $\V$. However, we are not interested in such a general case. We assume that $\V$ is an unbounded analytic set in $\CC^N$ for $N\ge 2$, and $f:\V\rightarrow \CC^M$ is a holomorphic mapping on $\V$. Recall that any bounded analytic set in $\CC^N$ is finite. 

\subsection{Growth exponent and polynomial mappings on analytic sets} \label{6.1} Under the above assumptions, the {\it growth exponent of function $f$ on $\V$} is defined by
\[ \rf :=\limsup_{\V\ni z\rightarrow \infty} \frac{\log \|f(z)\|}{\log \|z\|}.\] 
The growth exponent can be equal to zero (e.g., for $f\equiv\:$const$\ne 0$) and even to $-\infty$ (for $f\equiv0$). One can easily show that with the notation \ $\inf\: \emptyset:=+\infty$, \ we have
\[ \rf \!= \inf\{d\!\!\in\!\RR\: : \: \exists\, C\!>\!0 \ \ \forall z\!\in\! \V \ \ \|f(z)\| \le C(1+\|z||)^d\}  = \inf\!\left\{\!\!d\!\!\in\!\RR\, : \, \limsup_{\V\ni z\rightarrow \infty} \frac{\|f(z)\|}{\|z\|^d}\!<\!\infty\!\right\}\!\!.\]

We say that a function $f: \V \rightarrow \CC^{M}$ is a {\it polynomial mapping on $\V$} if there exists a function $F=(F_1,...,F_M):\CC^{N} \rightarrow \CC^{M}$ such that every its component $F_j$ is a polynomial of $N$ variables and $F_{|_V} = f$. Observe that in higher dimensions, in contrast to the one-dimensional case, a non-constant polynomial mapping is not necessarily proper.

\vskip 1mm

The image of an algebraic set under a proper polynomial mapping is an algebraic set. Indeed, being proper, the mapping is closed and so the statement is a consequence of Chevalley's Theorem (see e.g. \cite[VII.8.3]{Loj}, Proposition 2 and the unnumbered corollary that comes after the Chevalley Theorem in \cite{Loj}).

\vskip 1mm

If $f\!=\!(f_1,...,f_M)$ is a polynomial mapping on $\V\!=\!\CC^N$, then 
\[ \rf = \max\{{\rm deg}\,f_j\, :\, j =1,...,M\}=:{\rm deg}\,f.\] 
 Moreover, it is easy to show that for an unbounded algebraic set $\A\subset \CC^N$ and a polynomial mapping \ $f:\A\rightarrow \CC^M$ 
\be \label{osz_rfA} 
\rfA =\max_{j=1,...,M} \!\!\ro(f_j,\mathcal{A}) \le \max\{{\rm deg}_\A f_j \: : \: j=1,...,M\}, 
\ee
where \ ${\rm deg}_{\A} f_j:= \min\{ {\rm deg}\, F_j \: : \: F_j\in \CC[z_1,...,z_N], \ {F_j}_{|_{\A}}=f_j \}$. \ It is worth noticing that in the general case the growth exponent of a polynomial mapping on an algebraic set does not have to be an integer. 

\begin{example} Consider $\A=\{(w,z)\in \CC^2\: : \: w=z^2\}$ and $f(w,z) =(z^2,wz)$. We have \
${\rm deg}_{\A} f_1=1$, ${\rm deg}_{\A} f_2=2$ but $\rfA=3/2$.
\end{example}

\vskip 1mm

We will state the following property without a proof. This property is a slightly reformulated result of Bj\"ork \cite{Bjk}, and can also be seen as a consequence of Serre's Algebraic Graph Theorem (see e.g.\cite[VII.16.3]{Loj}). 

\begin{proposition} \label{equiv_polyn_map}
	Let $\A\subset \CC^N$ be an algebraic set of positive dimension and let  $f:\A\rightarrow \CC^M$ be a holomorphic mapping on $\A$. 
	Then
	\[ \rfA\!<\!\infty\: \ \ \  \Longleftrightarrow \ \ \ f \mbox{ is a polynomial mapping on} \ \A.\]
\end{proposition}

\subsection{Bernstein-Walsh inequality on algebraic sets}
The Siciak extremal function on algebraic sets can be expressed in an alternative way using growth exponents.
If $\mathcal{A}\subset\mathbb{C}^N$ is an algebraic set, then by $\mathcal{O}(\mathcal{A})$ we will denote the family of all holomorphic functions on $\mathcal{A}$ with values in $\CC$. We define the \emph{Siciak extremal function} of a compact set $K\subset\mathcal{A}$ by the formula
\begin{equation}\label{SEF}
\Phi_{K,\mathcal{A}}(z)=\sup\{|p(z)|^{1/\ro(p,\mathcal{A})}\,:\,
p\in\mathcal{O}(\mathcal{A}),\ \Vert p\Vert_K\leq1,\text{ and }0<\ro(p,\mathcal{A})<\infty\},
\end{equation}
for all $z\in\mathcal{A}$.
In view of Proposition \ref{equiv_polyn_map} we can use the family $\mathbb{C}[z]$ of all polynomials of $N$ variables in place of $\mathcal{O}(\mathcal{A})$ in (\ref{SEF}), but as stated, the definition emphasizes that the concept can be seen as intrinsic to $\mathcal{A}$. 

Clearly, if  $\mathcal{A}=\mathbb{C}^N$, then  $\Phi_{K,\mathbb{C}^N}\equiv\Phi_K$. Note also that 
\[
\Phi_{K,\mathcal{A}}\equiv\Phi_{\hat{K},\mathcal{A}},
\]
where $\hat{K}$ denotes the polynomially convex hull of $K$. Consequently, $\Phi_{K,\mathcal{A}}>1$ outside of $\hat{K}$.

\begin{proposition}\label{compare_Green_fns}
If $\mathcal{A}\subset\CC$ is a pure dimensional algebraic set of positive dimension and $K\subset\mathcal{A}$ is compact, then
\begin{equation}\label{compare_Green_fns_formula}
V_K(z)=\log\Phi_K(z)=\log\Phi_{K,\mathcal{A}}(z),\qquad z\in\mathcal{A}.
\end{equation}
\end{proposition}

\begin{proof} Note that from \cite{Strz} we have
\[ \ro\!(p,\A) = \min\{d\ge0\: : \: \exists\, C\!>\!0 \ \ \forall z\!\in\! \A \ \ \|p(z)\| \le C(1+\|z||)^d\} \]
for $p\in\CC[z]$.
The first equality in (\ref{compare_Green_fns_formula}) is the statement of the Siciak-Zaharjuta formula (see (\ref{Siciak-Za})). When combined with (\ref{osz_rfA}) and Theorem \ref{VKV=VK}, it implies the following:
\begin{eqnarray*}
V_K(z)=\log\Phi_K(z)&=&\sup\left\{\frac{1}{\deg p}\log^+|p(z)|\,:\,
p\in\mathbb{C}[z],\ \Vert p\Vert_K\leq1,\text{ and }\deg p>0\right\}\\
&\leq&\sup\left\{\frac{1}{\ro\!(p,\mathcal{A})}\log^+|p(z)|\,:\,
p\in\mathbb{C}[z],\ \Vert p\Vert_K\leq1,\text{ and } \ro\!(p,\mathcal{A})>0\right\}\\
&=&\log \Phi_{K,\mathcal{A}}(z)\leq V_{K,\mathcal{A}}(z)=V_K(z),
\end{eqnarray*}
for all $z\in\mathcal{A}$.
\end{proof} 

As a direct consequence we get the following inequality:

\begin{corol}\label{BWE}\textbf{(Bernstein-Walsh Inequality)} If $K$ is a compact subset of a pure dimensional algebraic set $\mathcal{A}$ of positive dimension
and $p\in\mathcal{O}(\mathcal{A})$, then
\begin{equation}\label{BWE_formula}
|p(z)|\leq\Vert p\Vert_K\exp\!\Big[\ro\!(p,\mathcal{V})\:V_K(z)\Big],\qquad z\in\mathcal{A}.
\end{equation}
\end{corol}

Note that $\ro\!(p,\mathcal{A})$ may be smaller than the degree of the polynomial extension of $p$ to the underlying vector space.
\begin{example}
Assume that $\mathcal{A}=\{(x,y)\in\mathbb{C}^2\,:\,x^2+y^2=1\}$ and $p(x,y)=x^3+x^2y+xy^2+y^3$. Then 
$1=\,\ro\!(p,\mathcal{A})<\,\ro\!(p,\mathbb{C}^2)=\deg p=3$.
\end{example}
In such cases the above version of the Bernstein-Walsh inequality is stronger than the classical one.

\subsection{Liouville's theorem on analytic sets}   
In what follows we often need to assume that the image of an analytic set under a holomorphic mapping is unbounded. We give a brief discussion of this assumption.

Let $\V$ be an analytic set in $\CC^{N}$ and   $f:\V\rightarrow \CC^M$ \  be a holomorphic mapping on $\V$. We are interested in choices of $\V$ and $f$ for which $f(\V)$ is not bounded. An elegant sufficient condition is related to a version of the Liouville theorem on $\V$. 

\vskip 1mm

We will say that an analytic set $\V\subset \CC^N$ has the {\it Liouville property} if any bounded function, which is holomorphic on $\V$, is constant. Obviously, the conectedness of $\V$ is a necessary 
condition for the Liouville property to be satisfied.

\vskip 1mm

Białożyt, Denkowski, Tworzewski and Wawak have obtained the following result in \cite{BDTW} (c.f. \cite{Cynk}).

\begin{thm}
An irreducible analytic set $\V\subset\CC^N$ of dimension $\ell$ has the Liouville property if the natural projection \ $\V\rightarrow Y$ \ into an $\ell$-dimensional vector space $Y\subset \CC^N$ is proper.  	
\end{thm}

As an immediate consequence of Theorem \ref{Rudin_Sadullaev}, we obtain
\begin{corol}
	Any irreducible algebraic set has the Liouville property. 
\end{corol}

\subsection{Order of functions defined on analytic sets} Let $\V$ be an analytic set in $\CC^{N}$ of positive dimension, and   let $f:\V\rightarrow \CC^M$  be a holomorphic mapping on $\V$ such that $f(\V)$ is an unbounded set. Taking into account the results given in the previous subsection, it is sufficient to assume that $f$ is not constant and $\V$ has the Liouville property. In particular, we can take any algebraic set $\V$ and a function $f$ that is not constant on some irreducible component of $\V$. 

\vskip 1mm

The {\it order of function} $f$ {\it on $\V$}, denoted by $\rfo$, \ is the growth exponent of multifunction $f^{-1}$ on $f(\V)$, i.e.,  
\[ \rfo :=\limsup_{f(\V)\ni w\rightarrow \infty} \frac{\log \|f^{-1}(w)\|}{\log \|w\|},\]
where $\|f^{-1}(w)\|:=\sup\{\|z\|\: : \: z\in f^{-1}(w)\}$. 

\vskip 2mm

Some basic properties of $\rfo$ are listed below. 

\begin{proposition} \label{wlasnosci_rfV}
	Under the above assumptions, the order of function $f$ has the following properties:
	\begin{enumerate}
		\item[(a)] $\rfo\ge 0$, 
		\vskip 4mm
		\item[(b)] \ if $\rfo<\infty$, then $f$ is a proper mapping,
		\vskip 4mm
		\item[(c)] \ $\frac1{\ro (f,\V)} \le \rfo$,
		\vskip 2mm
		\item[(d)] $\rfo = \inf\{d>0\: : \: \limsup\limits_{f(\V)\ni w\rightarrow \infty} \frac{\|f^{-1}(w)\|}{\|w\|^d}<\infty\}=\limsup\limits_{f(z)\rightarrow \infty, \: z\in\V} \frac{\log \|z\|}{\log \|f(z)\|}$.
	\end{enumerate}
\end{proposition}

\begin{proof} To prove {\it (a)} suppose otherwise. Then, for some $\varepsilon>0$ and for all $w\in f(\V)$ whose norm is sufficiently large, we would have $\|f^{-1}(w)\|< \|w\|^{-\varepsilon}$. Therefore, for $w\in f(\V)$ tending to infinity we could choose $z\in \V$ such that $f(z)=w$ and $\|z\|\le \|f(z)\|^{-\varepsilon}$. It would follow that $\V\ni z\rightarrow 0\in \V$ ($\V$ is closed) and $f(z)\rightarrow \infty$, which is a contradiction. 
	
	If $\rfo\!<\!\infty$, then for $\ro\:=1+\!\!\!\rfo>0$, and some $r\!\ge\! 1$ we have 
	 $\|f^{-1}(w)\| \le \|w\|^\ro $ \
	whenever  \ $w\in\! f(\V)$, $\|w\|\ge r$. In particular, \ $\|z\|\le \|f(z)\|^\ro$, provided that $z\in\! \V$ and $\|f(z)\|\ge r$. For $R\ge r$ such that a fixed compact set $K\subset f(\V)$ is contained in $\{\|w\|\le R\}$, we have 
	\[ f^{-1}(\{\|w\|=R\}) \subset \{\|z\|\le R^\ro\}\cap \V.\]
	 By maximum principle for plurisubharmonic functions on analytic sets, see e.g. \cite[Thm.1.3]{Sad83}, we get
	\[ \|f^{-1}(K)\|=\sup_{z\in f^{-1}(K)} \|z\| \le \sup_{z\in f^{-1}(\{\|w\|\le R\})} \|z\| = \sup_{z\in f^{-1}(\{\|w\|=R\})} \|z\| \le R^\ro, \] 
	which yields the desired conclusion in {\it (b)}. 

Observe that inequality {\it (c)} holds for $\rfo=+\infty$. Thus we can assume that \linebreak $\rfo<\infty$, and by property {\it (b)}, $f$ is proper. Therefore, $z\rightarrow\infty$ implies $f(z)\rightarrow\infty$. Take $\ro\: >\rfo$. For $w\in f(\V)$ outside some ball, we have \ $\frac{\log \|f^{-1}(w)\|}{\log\|w\|} \le \: \ro$. \ Consequently, \ $\frac{\log \|z\|}{\log\|f(z)\|} \le \: \ro$ \ for $z\in \V$ such that $\|f(z)\|$ is large enough. It follows that
\[ \frac1{\ro (f,\V)} \le \liminf_{\V\ni z\rightarrow \infty} \frac{\log\|z\|}{\log\|f(z)\|} \le \limsup_{\V\ni z\rightarrow \infty } \frac{\log\|z\|}{\log\|f(z)\|} \le \limsup_{f(z)\rightarrow \infty, \: z\in\V } \frac{\log\|z\|}{\log\|f(z)\|} \le \: \ro. \]  
Letting $\ro\: \rightarrow \rfo$ gives {\it (c)}.

Taking into account {\it (a)}, property {\it (d)} can be easily verified. 
\end{proof}

\vskip 2mm

\subsection{The Łojasiewicz exponent on analytic sets} 
Consider an analytic set $\V\subset \CC^N$ of positive dimension and a holomorphic mapping $f:\V\rightarrow \CC^M$ on $\V$. The {\it Łojasiewicz exponent of the function $f$ (at infinity) on $\V$} is defined by setting
\[ \lfV:=\sup\{d\in \RR\: : \: \mbox{there exist } C, R>0 \ \mbox{such that} \ C \|z\|^d < \|f(z)\| \ \mbox{for all} \ z\in \V, \ \|z\|>R\}. \]
We shall use the notation $\sup\emptyset = -\infty$. By elementary arguments, one can prove that 
\be \label{wzor_lfV} \lfV = \liminf_{\V\ni z\rightarrow \infty} \frac{\log \|f(z)\|}{\log \|z\|} = \sup \left\{d\in\RR\, : \, \liminf_{\V\ni z\rightarrow \infty} \frac{\|f(z)\|}{\|z\|^d}>0 \right\}. \ee
Consequently,
\be \label{lfV<rf} \lfV \le \rf.  \ee


The Łojasiewicz exponent can be a negative number, as shown in the following example. If $\V=\CC^2$ and $f(w,z)= (w,wz-1)$, then  $\lfV=-1$ (see \cite[Ex.3.1]{Kra07}). To check this, one can use either \cite[Thm.3.2]{Kra07} or \cite[Thm.1]{ChK}.

\begin{thm} \label{rfV_lfV} 
	Let  $\V\subset \CC^N$ be an analytic set of positive dimension and let $f:\V\rightarrow \CC^M$ be an unbounded holomorphic mapping on $\V$. Then we have the following implications.
	\begin{enumerate}
		\item[(a)] If \ $\lim\limits_{\V\ni z\rightarrow\infty} f(z) = \infty $, \ \ then \ $ \rfo\ge \frac1\lfV. $ \\
		In particular, this inequality holds if $f$ is a proper mapping on $\V$.
		\vskip 2mm
		\item[(b)] If \ $\lim\limits_{f(z)\rightarrow\infty, \: z\in\V} z = \infty $, \ \ then \ $ \rfo\le \frac1\lfV. $ \\
		In particular, this inequality holds if $f$ is a polynomial mapping on $\V$.
	\end{enumerate}
	\end{thm}

\begin{proof} Observe that by Proposition \ref{wlasnosci_rfV} and the assumption of {\it (a)}, we have
	\[\rfo = \limsup\limits_{f(z)\rightarrow \infty, \: z\in\V} \frac{\log \|z\|}{\log \|f(z)\|} \ge \limsup\limits_{z\rightarrow \infty, \: z\in\V} \frac{\log \|z\|}{\log \|f(z)\|} = \frac1{\liminf\limits_{z\rightarrow \infty, \: z\in\V} \frac{\log \|f(z)\|}{\log \|z\|}}=\frac1{\lfV}.\]
	The last equality is a consequence of (\ref{wzor_lfV}).  
		
	The inequality in {\it (b)} can be justified similarly to that in {\it (a)}. Moreover, if $f$ is a polynomial mapping on $\V$ then $f$ is the restriction to $\V$ of a polynomial mapping $F$ on $\CC^N$. By the Liouville estimate (see e.g. \cite[C.1.8]{Loj}), \ $\|F(z)\| \le M(1+\|z\|)^k$ \ for all $z\in\CC^N$ and some $M,k>0$. This implies that \ $f(z)\rightarrow\infty$ \ if \ $z\rightarrow\infty$ \ on $\V$.
\end{proof}	

\begin{corol} \label{rfV=lfV}
	If $f:\V\rightarrow \CC^M$ is a proper polynomial mapping on an analytic set $\V\subset \CC^N$ of positive dimension, then the set $f(\V)$ is also analytic of the same dimension as $\V$ and
	\[ \lfV \ = \ \frac1\rfo.  \]
\end{corol}

\vskip 2mm

Chądzyński and Krasiński proved in 1997 the following result (see \cite[Cor.2]{ChK}).

\begin{thm} \label{Ch_Kr} \ If $\A$ is an algebraic set of positive dimension in $\CC^N$, $N\ge 2$ and $ f: \A\rightarrow \CC^M$ \ is a polynomial mapping on $\A$, then 
	\[ \lfA >0 \ \ \ \ \ \Longleftrightarrow \ \ \ f \ \mbox{is a proper mapping on } \ \A.  \]
\end{thm}  

Taking into account $(b)$ in Proposition \ref{wlasnosci_rfV} and Corollary  \ref{rfV=lfV}, we get the following important consequence of the Chądzyński and Krasiński result.

\begin{corol}  \label{str_14} 
	Under the assumption of the above theorem, if $f$ is not bounded on $\A$, then 
	\[ \rfa < \infty \ \ \ \ \ \Longleftrightarrow \ \ \ f \ \mbox{is a proper mapping on } \ \A.  \]
\end{corol}

\section {Holomorphic transformations of pluricomplex Green functions}

Recall that the pluricomplex Green functions on analytic sets which are not algebraic, are never locally bounded above (see Theorem \ref{lok_bound_VK}). Therefore, in this section we consider only algebraic sets $\A\subset\CC^N$ and holomorphic mappings $f$ defined on $\A$. We are only  interested in non-trivial cases, therefore we  assume that $\A$ is of positive dimension and $f$ is not bounded. We shall estimate $V_{K} \circ f$ on $\A$ from below and from above for a compact set $K\subset f(\A)$. 

\subsection{Estimate of \ $V_{K} \circ f$ \ from below} Let $\A\subset\CC^N$ be an irreducible algebraic set of positive dimension and $f: \A \rightarrow \CC^{M}$ \ be a proper polynomial mapping. We already know that the image $f(\A)$ is an algebraic set (see the beginning of subsection 6.1). Moreover, by Theorem \ref{proper-map-thm}, dim$\,\A={\rm dim}\, f(\A)$ 
and $f(\A)$ is irreducible. Recall that an algebraic set is irreducible if and only if it is irreducible when regarded as an analytic set (see e.g. \cite[VII.11.1]{Loj}). Since any irreducible analytic set of positive dimension is unbounded, the order of $f$ is well defined on $\A$ and it is finite (see Corollary \ref{str_14}). 

\begin{thm} \label{VK_z_dolu}
	Let $\A\!\subset\! \CC^{N}$ be an irreducible algebraic set of positive dimension, $f\!:\A\rightarrow \CC^M$ \ be a proper polynomial mapping and $\B:=f(\A)$. Assume that $\B$ is locally irreducible. Then, for every compact set $K\subset \B$, 
	\be \label{nier_VK_z_dolu}
	V_{f^{-1}(K),\A} \ \le \ \ \rfb \ \ V_{K,\B}\circ f \ \ \ \ \ \ on \ \A,
	\ee
	which we can write equivalently as
	\be \label{nier_VK_z_dolu_krotko}
	\lfA \ \ V_{f^{-1}(K)} \ \le \ \ V_{K}\circ f \ \ \ \ \ \ on \ \A.
	\ee
\end{thm}

\begin{proof}
	By Siciak-Zaharjuta formula and Theorem \ref{VKV=VK}, the pluricomplex Green function $V_{\!\!\!f^{\!-1}\!(K),\A}$ can be written as
	\begin{align*} 
	V_{f^{-1}(K),\A}(z) &= V_{f^{-1}(K)}(z)  = \log \sup\left\{|p(z)|^{1/{\rm deg}\, p}  : \: p\in
	\CC[z], \ {\rm deg}\, p\ge 1, \max\limits_{z\in f^{-1}(K)} |p(z)| = 1
	\right\} \\
	&= \sup\left\{\frac1{{\rm deg}\, p} \log^+ |p(z)| \: : \: p\in
	\CC[z], \ {\rm deg}\, p\ge 1, \ \max\limits_{z\in f^{-1}(K)} |p(z)| = 1
	\right\}
	\end{align*}
	 for $z\in\A$. Fix $p\in
	 \CC[z]$ such that ${\rm deg}\, p\ge 1$ and $\max\limits_{z\in f^{-1}(K)} |p(z)| = 1$. The function 
	 \[ u:=\tfrac1{{\rm deg}\, p} \: \log^+ |p| \]
	 is plurisubharmonic and continuous in $\CC^N$. Moreover, $u\in\mathcal{L}(\CC^N)$. Set 
	 \[v(w):=\max \ u(f^{-1}(w))=\max \{u(z)\: : \: f(z)=w\} \ \  \ \ {\rm for} \ \ w\in \B. \] By Lemma \ref{basic_lemma}, $v$ is plurisubharmonic on $\B$. We will first show that $v\in \mathcal{L}(\B)$.
	Since the order of $f$ is finite, we fix an arbitrary $\ro\:>\rfb$. For $w\in \B$, which is sufficiently far from zero and such that $\|f^{-1}(w)\|\le \|w\|^\ro$, we have
	\begin{align*}
	\tfrac1\ro \:v(w) &=\tfrac1\ro \max \{u(z)\: : \: z\in f^{-1}(w)\} \le \tfrac1\ro \max \{c+ \log(1+\|z\|) \: : \: z\in f^{-1}(w)\}\\ 
	&\le \tfrac1\ro \left[c+ \log(1+\|f^{-1}(w)\|)\right] \le \tfrac1\ro \left[c+ \log(1+\|w\|^\ro)\right]\\ &\le \tfrac1\ro \left[c+ \log(2(1+\|w\|)^\ro)\right]
	= \tfrac{c}\ro + \tfrac{\log 2}\ro + \log(1+\|w\|), 
	\end{align*}
where $c=c(u)$ is from the definition of $\mathcal{L}(\CC^N)$. This gives us an estimate from above of $\tfrac1\ro\,v(w)$ for all $w\in \B$ with a sufficiently large $\|w\|$. We can obtain an analogous estimate on the whole $\B$, because $v$ is locally bounded on $\B$ as a plurisubharmonic function. Thus $\tfrac1\ro\, v\in\mathcal{L}(\B)$.

Observe that, for $w\in K$,
\[ v(w) \le \max \{u(z)\: : \: z\in f^{-1}(K)\} = \max \{\tfrac1{{\rm deg}\, p} \: \log^+ |p(z)| \: : \: z\in f^{-1}(K)\} =0, \] 	
and consequently, 
\[ \tfrac1\ro\: v\le V_{K,\B} \ \ \ \ {\rm on} \ \ \B.\]	
We now fix $z\in\A$ and note that
\[ \tfrac1\ro \, u(z) \le \tfrac1\ro\, \max u(f^{-1}(f(z))) =\tfrac1\ro\, v(f(z)) \le  V_{K,\B} (f(z)), \]
which yields the inequality
\[ \tfrac1\ro \, V_{f^{-1}(K)} \le V_{K,\B} \circ f \ \ \ \ {\rm on} \ \ \A.\]
By letting $\ro$ tend to $\rfb$, in view of Theorems \ref{VKV=VK} and \ref{rfV_lfV}, we obtain the inequalities (\ref{nier_VK_z_dolu}) and (\ref{nier_VK_z_dolu_krotko}). 
\end{proof}

As a consequence, we get the following statement. 

\begin{corol}
	Under the assumptions of Theorem \ref{VK_z_dolu}, the order of mapping $f$ on $\A$ is positive and the Łojasiewicz exponent of $f$ on $\A$ is finite.  
\end{corol}

Theorem \ref{Ch_Kr} shows that the Łojasiewicz exponent of $f$ is positive if $f$ is polynomial and proper on $\A$. Now, we will show that $\lfA$ is the maximal positive constant $d$ in   inequality (\ref{VK_z_dolu_d}). Recall that $B(b,\delta):=\{w\: : \: \|w-b\|\le \delta\}$ and $\bar{B}(b,\delta)$ is the closed ball.  

\begin{proposition} \label{converse_VK_z_dolu}
	Let $\A\!\subset\! \CC^{N}$ be an irreducible algebraic set of positive dimension and let $f\!:\A\rightarrow \CC^M$ \ be a polynomial mapping such that $\B:=f(\A)$ is an  irreducible algebraic set. Assume that for  certain $b\in\B$ and $\delta>0$ the preimage of $K:=\B \cap \bar{B}(b,\delta)$ is compact in $\A$ and the inequality 
	\be \label{VK_z_dolu_d} d \ V_{f^{-1}(K),\A} \ \le \ V_{K,\B} \circ f \ \ \ \ holds \ on \ \ \A
	\ee
	for some $d>0$. Then \ $\lfA\ge d$. In consequence, $f$ is proper and 
	\begin{align*} 
	\lfA \!&=\! \max\{d\!>\!0 \: : \: \mbox{inequality (\ref{VK_z_dolu_d}) holds with} \ K\!=\!\B\cap\bar{B}(b,\delta) \ \mbox{for some} \ b\!\in\! \B, \ \delta\!>\!0\} \\ 
	&=\max\{d\!>\!0 \: : \: \mbox{inequality (\ref{VK_z_dolu_d}) holds for all compacts} \ K\subset \B\}. 
	\end{align*}
\end{proposition}

\begin{proof} For the sake of simplicity, we may suppose that $0\in \B$ and take $b=0$.
	Since $\B$ is an  irreducible algebraic set,  we can find -- as it was done in Example \ref{nplp_on_A} -- a positive number $r$ such that
	\[ V_{K,\B}(w) \le \log ^+ \tfrac{\|w\|}r,  \]
	 for all $w\in \B$.
	For some positive $R$, we have $f^{-1}(K)\subset \A\cap\bar{B}(0,R)$. Fix $w\in\B$ and $z\in f^{-1}(w)$. In view of our assumptions, we have
	\[ \log^+ \tfrac{\|z\|}{R} = V_{\bar{B}(0,R)}(z) \le V_{\A\cap\bar{B}(0,R)}(z) = V_{\A\cap\bar{B}(0,R),\A}(z) \le V_{f^{-1}(K),\A} (z) \le \tfrac1d V_{K,\B} (w) \le \tfrac1d \log^+ \tfrac{\|w\|}{r}.  \]
	Consequently, for $z\in f^{-1}(w)$ far enough from zero, we get \ $\tfrac{\|z\|^d}{R^d} \le \tfrac{\|f(z)\|}{r}$ \ 
	and 
	\[ \lfA = \liminf_{\A\ni z\rightarrow \infty} \frac{\log \|f(z)\|}{\log \|z\|} \ge \liminf_{\A\ni z\rightarrow \infty} \frac{\log r -d\,\log R +d\, \log\|z\|}{\log \|z\|} =d.	\]
Therefore, $f$ is proper (see Theorem \ref{Ch_Kr}). This finishes the proof. \end{proof}

A similar result holds if we take the intersection of an Euclidean ball with $\A$ instead of $\B$. 

\begin{proposition} \label{converse_VE_z_gory}
	Let $\A\!\subset\! \CC^{N}$ be an irreducible algebraic set of positive dimension and let  $f\!:\A\rightarrow \CC^M$ \ be a polynomial mapping with discrete fibres such that $\B:=f(\A)$ is a locally irreducible algebraic set and \ dim$\,\A = \ \,$dim$\,\B$. \ If, for some $a\in\A$ and $\delta>0$,  the inequality 
	\be \label{VE_z_gory_d} d \ V_{E,\A} \ \le \ V_{f(E),\B} \circ f \ \textit{ holds  on }\A,
	\ee
	where $E:=\A \cap \bar{B}(a,\delta)$ and $d>0$ is a constant, then \ $\lfA\ge d$. Consequently, $f$ is proper. 
\end{proposition}

\begin{proof} We assume that $a=0\in \A$. The mapping $f$ is open (see Theorem \ref{open-map-thm}), and so we can find $r>0$ such that $B_r\subset f(\A \cap {B}(0,\delta)) \subset f(E)$, where $B_r$ is as in Example \ref{nplp_on_A}. Thus for $z\in\A$, we have
	\[ \log^+ \tfrac{\|z\|}{\delta} \le V_{\A\cap\bar{B}(0,\delta)}(z) = V_{E,\A}(z) \le \tfrac1d V_{f(E),\B} (f(z)) \le \tfrac1d V_{B_r}(f(z)) \le \tfrac1d \log^+ \tfrac{\|f(z)\|}{r}.  \] 
	Consequently, for $z\in \A$ far enough from zero, we get \ $ \tfrac{\|z\|^d}{\delta^d} \le \tfrac{\|f(z)\|}{r} $ \ 
	and this implies that $f$ is proper. 
\end{proof}

\vskip 3mm

\subsection{Estimate of $V_{K} \circ f$ from above} We will continue to consider an irreducible algebraic set $\A\subset\CC^N$ of positive dimension. By Proposition \ref{equiv_polyn_map}, a holomorphic mapping $f\!:\A\rightarrow \CC^M$ is a polynomial map on $\A$ if and only if the growth exponent $\rfA$ is finite.

\begin{thm} \label{VK_z_gory}
	Let $\A\!\subset\! \CC^{N}$ be an irreducible algebraic set of positive dimension, $f\!:\A\rightarrow \CC^M$ \ be a proper polynomial mapping and $\B:=f(\A)$.  Then, for every compact set $K\subset \B$, 
	\be \label{nier_VK_z_gory}
	 V_{K,\B}\circ f \ \le \ \ \rfA \ V_{f^{-1}(K),\A} \ \textit{  on } \A.
	\ee
\end{thm}

\begin{proof}
Fix $d>\rfA$ and $u\in \mathcal{L}(\A)$ such that $u_{|_K}\le 0$. Applying Proposition \ref{a,b}, we get $u\circ f\in PSH(\A)$. By the definition of the growth exponent, we can find a positive constant $C$ such that $\|f(z)\| \le C + \|z\|^d$ for all points $z\in \A$. Since $u\in \mathcal{L}(\A)$, there exists a constant $c$ depending only on $u$ such that, for $z\in\A$, we have 
\begin{align*}
 \tfrac1d (u\circ f) (z) &\le \tfrac1d [c+\log(1+\|f(z)\|)] \le \tfrac{c}d   +\tfrac1d \log(1+C+\|z\|^d)  \\ 
&< \tfrac{c}d   +\tfrac1d \log(1+C) + \tfrac1d \log(1+ \|z\|^d) \le \tilde{c} + \log (1+\|z\|),
\end{align*}
where $\tilde{c}=\tilde{c}(u,f)$. Consequently, $\tfrac1d (u\circ f)\in\mathcal{L}(\A)$. Moreover, 
\[ \tfrac1d (u\circ f)_{|_{f^{-1}(K)}} = \tfrac1d\: u_{|_K}\le 0,\] and this yields the inequality
\[ \tfrac1d (u\circ f) \le V_{f^{-1}(K),\A}  \ \textrm{ on } \A.\]
Now we get inequality (\ref{nier_VK_z_gory}) by taking the supremum over all $u$ and letting $d$ tend to $\rfA$. 
\end{proof}

\vskip 2mm

The converse is true in the following sense.

\begin{proposition} \label{converse_VK_z_gory}
	Let $\A\!\subset\! \CC^{N}$ be an irreducible algebraic set of positive dimension and let $f\!:\A\rightarrow \CC^M$ \ be a proper holomorphic mapping. Let $\B:=f(\A)$, 
	$b\in\B$, $\delta>0$ and $K:=\B \cap \bar{B}(b,\delta)$. If, for some $d>0$,  the inequality 
	\be \label{VK_z_gory_d}  V_{K,\B}\circ f \ \le \ \ d \ \: V_{f^{-1}(K),\A} \ \textit{ holds  on } \A,
	\ee
	then $f$ can be extended to a polynomial mapping on $\CC^N$ and $\rfA\le d$. \ In consequence, $\B$ is an irreducible algebraic set and
	\[ \rfA \!=\! \min\{d\!>\!0 \: : \: \mbox{inequality (\ref{VK_z_gory_d}) holds with} \ K\!=\!\B\cap \bar{B}(b,\delta) \ \mbox{for some} \ b\!\in\! \B, \ \delta\!>\!0\} \] 
	\[ = \min\{d\!>\!0 \: : \: \mbox{inequality (\ref{VK_z_gory_d}) holds for all compacts} \ K\subset\B\}. \] 
\end{proposition}

\begin{proof} Taking into account Proposition \ref{equiv_polyn_map}, it is sufficient to prove the estimate $\rfA\le d$. \  Let $E:=f^{-1}(K)$. Observe that $E$ is compact in $\A$ and $f^{-1}(\B\cap B(b,\delta))$ is open in $\B$. We can find $a\in \A, \ \rho>0$ such that $\A \cap \bar{B}(a,\rho) \subset f^{-1}(\B\cap B(b,\delta))\subset f^{-1}(K)$. Without losing any generality, we can assume that $a=0$. As in Example \ref{nplp_on_A}, we have \ $ V_{\A \cap \bar{B}(0,\rho),\A}(z) \le \log^+ \tfrac{\|z\|}{r} $ \ for some $r>0$ independent of $z\in \A$. Observe that 
	\[ \log^+ \frac{\|f(z)-b\|}\delta = \left( V_{\bar{B}(b,\delta)} \circ f \right) (z) \le \left( V_{K} \circ f \right) (z) \le d \ \: V_{f^{-1}(K),\A} (z)  \le d \ \log^+ \tfrac{\|z\|}{r}.\]
		Consequently,  since $f$ is proper, $z\rightarrow \infty$ implies $f(z)\rightarrow \infty$ and thus \ $ \|f(z)\| \le \|b\| + \delta \tfrac{\|z\|^d}{r^d} $ \ 
	for \ $z\in \A$ \ far enough from zero. Therefore,
	\[ \rfA = \limsup_{\A\ni z\rightarrow \infty} \frac{\log \|f(z)\|}{\log \|z\|} \le 
	\limsup_{\A\ni z\rightarrow \infty} \frac{\log \left( \|b\| + \delta  \tfrac{\|z\|^d}{r^d}\right)}{\log \|z\|} = d,\]
	and this completes the proof.
\end{proof}

\begin{proposition} \label{converse_VE_z_dolu}
	Let $\A\!\subset\! \CC^{N}$ be an irreducible algebraic set of positive dimension and let $f\!:\A\rightarrow \CC^M$ \ be a proper holomorphic mapping. Let $\B:=f(\A)$, 
	$a\in\A$, $\delta>0$ and $E:=\A \cap \bar{B}(b,\delta)$. If, for some $d>0$,  the inequality 
	\be \label{VE_z_dolu_d}  V_{f(E),\B}\circ f \ \le \ \ d \ \: V_{E,\A} \ \textit{ holds  on } \A,
	\ee
	then $f$ can be extended to a polynomial mapping on $\CC^N$ and $\rfA\le d$. \ Moreover,
	\[ \rfA \!=\! \min\{d\!>\!0 \: : \: \mbox{inequality (\ref{VE_z_dolu_d}) holds with} \ E\!=\!\A\cap \bar{B}(a,\delta) \ \mbox{for some} \ a\!\in\! \A, \ \delta\!>\!0\} \] 
	\[ = \min\{d\!>\!0 \: : \: \mbox{inequality (\ref{VE_z_dolu_d}) holds for all compacts} \ E\subset\A\}. \]
	In particular, for any compact $E\subset \A$, the inequality \ $ V_{f(E),\B}\circ f \le \rfA \: V_{E,\A}$ \ holds on $\A$. 
\end{proposition}

\begin{proof} 
Since $f(E)$ is compact, we take $R>0$ such that $f(E)\subset B(0,R)$. This yields, for all $z\in \A$ and some $r>0$, the following estimates
\[ \log^+ \frac{\|f(z)\|}R \le \left( V_{f(E)} \circ f \right)\!\! (z) \le  \left( V_{f(E),\B} \circ f \right)\!\! (z) \le d \ V_{E,\A} (z) \le d \, \log^+ \tfrac{\|z-a\|}{r}.\] 
Consequently, \ $\|f(z)\| \le R\tfrac{\|z-a\|^d}{r^d}$ \ 
	for \ $z\in \A$ \ far enough from zero and \ 
	$ \rfA \le d$,  
	which proves that $f$ is a polynomial mapping on $\A$ (see Proposition \ref{equiv_polyn_map}). 
	
	Now, observe that by Theorem \ref{VK_z_gory}, inequality (\ref{VE_z_dolu_d}) holds for all compact sets $E\subset \A$ because we can take $K=f(E)$ that is a compact set in $\B$ and $E\subset f^{-1}(f(E))$. Hence, inequality (\ref{VE_z_dolu_d}) is satisfied with $d=\rfA$.
\end{proof}

\vskip 3mm

\subsection{Main theorems} 
Two first results given below are consequences of the theorems and propositions proved in the previous subsections. They generalize Theorem \ref{klasyka} and give additional properties of proper polynomial mappings on algebraic sets.

\begin{thm}
	Let $\A\subset\CC^N$ be an irreducible algebraic set of positive dimension  and let \linebreak  $f\!:\A\rightarrow \CC^M$ \ be a non-constant holomorphic mapping on $\A$. Then the following conditions are equivalent:
	\begin{itemize}
		\item[{\it (i)}] $f$ is a proper polynomial mapping on $\A$,
		\item[{\it (ii)}] 
		$\begin{cases}
			\rfA<\infty \\
			\lfA >0,
		\end{cases}$
	\end{itemize}
	Assume additionally that $f$ is proper and  $\B:=f(\A)$ is locally irreducible. Then $f$ is a polynomial map on $\A$ if and only if any of the following conditions holds:  
	\begin{itemize}
		\item[{\it (iii)}]
		\mbox{for all compact sets} \ $K\subset \B$ 
		\[\lfA \ V_{f^{-1}(K),\A} \ \le \ V_{K,\B}\circ f \ \le \ \ \rfA \ V_{f^{-1}(K),\A} \ \ on \ \A, \]
		\item[{\it (iv)}]
		$\mbox{for some} \ d>0 \ \mbox{and for all compact sets} \ K\subset\B$ \
		\[V_{K,\B}\circ f \ \le \ \ d \ V_{f^{-1}(K)),\A} \ \ on \ \A,	\]
		\item[{\it (v)}]
		$\mbox{for some} \ d>0 \ \mbox{and for all compact sets} \ E\subset\A$ \
		\[V_{f(E),\B}\circ f \ \le \ \ d \ V_{E,\A} \ \ on \ \A,	\]
		\item[{\it (vi)}]
		$\mbox{for some} \ d>0, \ b\in \B, \ \delta>0$ \
		\[ V_{\B\cap\bar{B}(b,\delta),\B}\circ f \ \le \ \ d \ V_{f^{-1}(\B\cap\bar{B}(b,\delta)),\A} \ \ on \ \A, \] 
		\item[{\it (vii)}]
		$ \mbox{for some} \ d>0, \ a\in \A, \ \delta>0$ \
		\[ V_{f(\A\cap\bar{B}(a,\delta)),\B}\circ f \ \le \ \ d \ V_{\A\cap\bar{B}(a,\delta),\A} \ \ on \ \A.	\] 
	\end{itemize}
\end{thm}

\begin{proof}
	The equivalence of the first two conditions is a consequence of Proposition \ref{equiv_polyn_map}, Theorem \ref{Ch_Kr} and Corollary \ref{str_14}.  By Theorems \ref{VK_z_dolu} and \ref{VK_z_gory}, conditions $(iii), \ (iv), \ (v), \ (vi), \ (vii)$ are a consequence of $(i)$. Theorem \ref{converse_VE_z_dolu} yields $(vii)$, which implies that $f$ is a polynomial mapping on $\A$. 
\end{proof}

\begin{thm} \label{VK_z_gory_z_dolu}
		Let $\A\!\subset\! \CC^{N}$ be an irreducible algebraic set of positive dimension and let \linebreak  $f\!:\A\rightarrow \CC^M$ \ be a proper polynomial mapping. Assume that $\B:=f(\A)$
		is locally irreducible. Then, for every compact set $K\subset \B$, 
		\be \label{nier_VK_z_gory_z_dolu}
		\lfA \ V_{f^{-1}(K),\A} \ \le \ V_{K,\B}\circ f \ \le \ \ \rfA \ V_{f^{-1}(K),\A} \ \ on \ \A.
		\ee
Moreover, the following conditions are equivalent:
		\begin{itemize}
			\item[{\it (i)}] \
			for some \ $d>0$ \ and for all compact sets \ $K\subset \B$ 
			\[ V_{K,\B}\circ f \ = \ \ d \: \ V_{f^{-1}(K),\A} \ \ on \ \A, \]
			\item[{\it (ii)}] \
			for some \ $d>0, \ b\in \B, \ \delta>0$  
			\[ V_{\B\cap\bar{B}(b,\delta)}\circ f \ = \ \ d \: \ V_{f^{-1}(\B\cap\bar{B}(b,\delta))} \ \ on \ \A, \]
			\item[{\it (iii)}] \ $ \lfA \ = \ \rfA ,	$ \vskip 2mm
			\item[{\it (iv)}] \ the limit \ $\lim\limits_{\A\ni z\rightarrow \infty} \frac{\log \|f(z)\|}{\log \|z\|}$ \ exists.
		\end{itemize}
\end{thm}
	 
\begin{proof}
	The first part of the statement is a consequence of Theorems \ref{VK_z_dolu} and \ref{VK_z_gory}. The definition of the growth exponent and (\ref{wzor_lfV}) give the equivalence of $(iii)$ and $(iv)$. In view of (\ref{nier_VK_z_gory_z_dolu}), $(iii)$ implies $(i)$. It is obvious that $(i)\Rightarrow (ii)$. By Propositions \ref{converse_VK_z_dolu}, \ref{converse_VK_z_gory} and (\ref{lfV<rf}), we obtain $(iii)$ from $(ii)$.
\end{proof}	

We can easily observe that 
either of conditions $(i)$ and $(ii)$ implies \ $d=\lfA$. Moreover, if there exists the limit in $(iv)$, then it is positive, finite and equal to $\lfA$.

\vskip 2mm 

Theorem \ref{generalization} given below is very closely related to Theorem \ref{klasyka}. It is formulated in such a way as to show a direct generalization of the latter result to the case of holomorphic mappings on algebraic sets. Due to a recent result obtained by Białożyt, Denkowski and Tworzewski \cite{BDT}, condition (iii) below can be stated analogously as in Theorem \ref{klasyka}. We present the conclusion of their theorem adapted to our situation: 

\begin{thm} \label{bdt}
If $\A\subset\CC^N$ is an irreducible algebraic set of positive dimension, \linebreak  $f=(f_1,...,f_M)\!:\A\rightarrow \CC^M$ \ is a non-constant holomorphic mapping on $\A$ and $f(\A)$ is an analytic set, then 
\[ \lfA=\rfA \]
if and only if $f$ is a proper polynomial mapping on $\A$, 	$\ro\!\!(f_1\!,\mathcal{A})=...=\ro\!\!(f_M\!,\mathcal{A})$, and
\be \label{cone} \bigcap_{j=1}^M f_j^{-1}(0)^* =\{0\},  \ee
where $\mathcal{D}^*$ is the cone at infinity of an algebraic set $\mathcal{D}\subset \CC^N$ of positive dimension, defined by
\[ \mathcal{D}^*:=\{z\in \CC^N \: : \: \exists (a_n)_{n=1}^\infty \subset \mathcal{D}, \ (t_n)_{n=1}^\infty\subset \CC : \|a_n\|\rightarrow\infty, \ t_na_n\rightarrow z\}. \] 
\end{thm} 
 
This cone corresponds to points at infinity of the set $\mathcal{D}$, and therefore condition (\ref{cone}) asserts that the sets $f^{-1}_j(0)$ do not have common points at infinity. Some properties of the cone at infinity of algebraic set are given in \cite{LP}. In particular, the authors show that it can be computed by means of Gr\"obner basis.  

In the case of a polynomial $p\in\CC[z_1,..,z_N]$, we have \[ \hat{p}^{-1}(0)=p^{-1}(0)^*. \]
Thus condition (\ref{cone}) is a generalization of (\ref{homogeneous}). As shown in Section \ref{6.1}, the growth exponent $\rfA$ is a generalization of the concept of the degree of a polynomial.

\begin{thm} \label{generalization}
	Let $\A\subset\CC^N$ be an irreducible algebraic set of positive dimension.  
	If \linebreak  $f=(f_1,...,f_M)\!:\A\rightarrow \CC^M$ \ is a non-constant holomorphic mapping on $\A$, $\B:=f(\A)$ is a locally irreducible analytic set and $k,\ell>0$, then the following conditions are equivalent:
	\vskip 2mm
	\begin{itemize}
		\item[(i)] $f$ is a polynomial mapping on $\A$ such that \ $\rfA\le \ell$ \ and \ $\lfA\ge k$, 
		\item[(ii)] $f$ is a proper mapping and, for every compact set $K\subset \B$, 
		\be \label{nier_uogolnienie} 
		k \ V_{f^{-1}(K),\A}\ \le \ V_{K,\B}\circ f \ \le \ \ell \ V_{f^{-1}(K),\A} \ \ \ \ \ on \ \ \A. 
		\ee 
	\end{itemize}	
	If in addition $k=\ell$, then either of two above conditions is equivalent to the following one
	\begin{itemize}
		\item[(iii)]  $f$ is proper, $\ro\!\!(f_1\!,\mathcal{A})=...=\ro\!\!(f_M\!,\mathcal{A})=\ell$, and
		\[ \bigcap_{j=1}^M f_j^{-1}(0)^* =\{0\}.  \] 
	\end{itemize} 	
\end{thm}

\begin{proof}
From (i), we know that  $\lfA>0$ for any polynomial mapping $f$, which implies that $f$ is proper. By Theorem \ref{VK_z_gory_z_dolu}, inequality (\ref{nier_VK_z_gory_z_dolu}) is satisfied, and consequently, so is (\ref{nier_uogolnienie}). Assume now (ii). Applying Proposition \ref{converse_VK_z_gory}, we conclude that $f$ is a polynomial mapping on $\A$ and $\rfA\le \ell$. Proposition \ref{converse_VK_z_dolu} gives $\lfA\ge k$.

Condition (ii) and inequality (\ref{lfV<rf}) imply that $k\le \lfA \le \rfA \le \ell$. If $k=\ell$, then $\lfA=\rfA=\ell$, so by Theorem \ref{bdt} and (\ref{osz_rfA}), we obtain (iii). To prove the converse, observe that $\ro\!\!(f_1\!,\mathcal{A})=...=\ro\!\!(f_M\!,\mathcal{A})=\ell$, provided that $\rfA=\ell<\infty$ and $f$ is a polynomial mapping on $\A$. In view of Theorem \ref{bdt}, we get $\lfA=\rfA=\ell=k$, and so condition (i) is satisfied. 
\end{proof}	

\subsection{Examples} The assumption of local irreducibility in Theorem \ref{VK_z_gory_z_dolu} is only used so that Lemma \ref{basic_lemma} gives plurisubharmonicity of the push-forward functions. If weak plurisubharmonicity is sufficient, the theorem still applies as illustrated in the next example.

\begin{example}
Let $\mathcal{A}=\{(w,z)\in\mathbb{C}^2\,:\,w^3=z^2\}$. Clearly $(0,0)$ is the only singular point of $\mathcal{A}$. Condider the mapping 
\[
f:\mathbb{C}\ni\xi\longmapsto(\xi^2,\xi^3)\in\mathcal{A}. 
\]
Obviousy $\varphi$ is proper, $\mathcal{L}_\infty(f|\mathbb{C})=\:\ro\!(f,\mathcal{A})=3$. We can easily check that $f$ is surjective. Indeed, if $(w,z)\in\mathcal{A}$, then $\sqrt{w}=\{\xi,-\xi\}$ for some $\xi\in\mathbb{C}$. Moreover
$(\pm\xi^3)^2=w^3=z^2.$ Since $\sqrt{z^2}=\{z,-z\}$, we can choose the sign so that $\xi^3=z$. Consequently, $\mathcal{A}$ is irreducible as the set $\reg{\mathcal{A}}$ is connected being the image of $\mathbb{C}\setminus\{0\}$ via $f$. Now, let 
\[
E=\{(w,z)\in\mathcal{A}\,:\,|w|\leq 1,\ |z|\leq 1\}. 
\]  
Since $\sing{\A}\subset E$, weak plurisubharmonicity of the push-forwards of the relevant functions is enough, the proof of Theorem \ref{VK_z_gory_z_dolu} works in this case and we get
\[
\log^+|z|= V_{E,\mathcal{A}}(w,z)=\frac 3 2 \log^+|w|,\qquad (w,z)\in\mathcal{A}.
\]
\end{example}

Sometimes the assumption of irreducibility can be dropped as shown in the following example.

\begin{example}
Let $p,q$ be two non-constant polynomials of one complex variable such that $p(0)=q(0)=0$. Consider $\mathcal{A}=\{(w,z)\in\mathbb{C}^2\,:\,wz=0\}$ and the proper polynomial mapping
\[
f:\mathcal{A}\ni(w,z)\longmapsto p(w)+q(z)\in\mathbb{C}.
\]
It is easy to see that
\[
\rfa=\lfA=\min(\deg p,\deg q)\quad\text{and}\quad \rfA=\max(\deg p,\deg q).
\]
In particular, if $d=\deg p=\deg q$ and $E=(p^{-1}(\Delta)\times\{0\})\cup(\{0\}\times q^{-1}(\Delta))$, where $\Delta$ denotes the closed unit disk, 
then
\begin{equation}\label{cross_set_Green}
V_{E,\mathcal{A}}(w,z)=\frac{1}{d}\log^+|p(w)+q(z)|,\qquad (w,z)\in\mathcal{A}.
\end{equation}
Indeed, while $\mathcal{A}$ is not irreducible, it is the union of two complex lines $\mathbb{C}\times\{0\}$ and $\{0\}\times\mathbb{C}$ and Lemma \ref{basic_lemma} applies to the restrictions of $f$ to these lines, i.e. to $p$ and $q$. 
If $u\in\mathcal{L}(\mathcal{A})$, then -- using the push-forward notation -- we have:
\[
\frac 1 d f[u]=\frac 1 d \max(p[u],q[u])\in\mathcal{L}(\mathbb{C}),
\]
which implies (\ref{cross_set_Green}), just like in Theorem \ref{VK_z_gory_z_dolu}.
\end{example}

\vskip 2mm

Our paper concludes with two straightforward examples in which we can use Theorem \ref{VK_z_gory_z_dolu} in order to obtain a Bernstein-Walsh polynomial inequalities on algebraic sets.

\begin{example}
The algebraic set 
\[
\mathcal{A}=\{(x,y,z)\in\mathbb{C}^3\,:\,x^2+y^2+z^2=1\}
\]
is a two-dimensional connected complex submanifold of $\mathbb{C}^3$. Let $f=(f_1,f_2):\mathbb{C}^2\longrightarrow\mathbb{C}^2$
be a non-constant polynomial mapping whose homogenous component of highest degree $\hat{f}$ satisfies the condition
$\hat{f}^{-1}((0,0))=\{(0,0)\}$. Let $\pi(x,y,z)=(x,y)$. Note that
\[
\mathcal{L}_\infty(f\circ\pi|\mathcal{A})=\rho(f\circ\pi,\mathcal{A})=\deg f.
\]
Consider the polynomial polyhedron 
\[
E=\{(x,y,z)\in\mathcal{A}\,:\,|f_1(x,y)|\leq 1,\ |f_2(x,y)|\leq 1\}=(f\circ \pi)^{-1}(\Delta \times \Delta).
\]
According to Theorem \ref{VK_z_gory_z_dolu}
\[
V_{E,\mathcal{A}}(x,y,z)=\frac{1}{\deg f}\max\{
\log^+|f_1(x,y)|,\log^+|f_2(x,y)|
\},\qquad (x,y,z)\in\mathcal{A}.
\]
In this case, the Bernstein-Walsh inequality tells us that for any non-constant polynomial $p$ we have the estimate
\[ 
|p(x,y,z)| \le \max\{|f_1(x,y)|,|f_2(x,y)|\}^{\ro(p,\A)/\deg f} \|p\|_E \ \ \ \ \mbox{for} \ \ \ (x,y,z)\in \A\setminus E.
\]
\end{example}

\begin{example}
Let \[ \A=\{(x,y,z)\in\CC^3\: : \: x^2+y^2+z^2=4, \ x^2+y^2=2x\}=
\{(x,y,z)\in\mathbb{C}^3\,:\,z^2=4-2x,\ y^2=2x-x^2\},\]  
\[E=\{(x,y,z)\in \CC^3 \: : \: x\in [0,2], \ x^2+y^2+z^2=4, \ x^2+y^2=2x\}.\] The set $E$ is usually cold the {\it Viviani's window} and is a real curve. Consider first 
\[ f: \CC \ni t \longmapsto (t, 2t-t^2,4-2t)\in \B:=\{(u,v,w)\in\CC^3\: : \: u^2+v+w=4, \ u^2+v=2u\}\]
and $K=f([0,2])$. Observe that $f$ is proper, \ $f(\CC)=\B$ and $\B$ is an irreducible algebraic set because is given by a polynomial parametrization (see e.g. \cite[Prop.$\:$4, Chap.$\:$4, $\mathsection$ 5]{CLO}). 
Moreover, \ $\ro\!\!(f\!,\CC)=\mathcal{L}_\infty(f|\CC)=4$ \ and \ $f^{-1}(K)=[0,2]$. \ 
Note that  the set 
\[
\mathcal{A}':=\{(x,y)\in\mathbb{C}^2\,:\,y^2=2x-x^2\}=\{(x,y)\in\mathbb{C}^2\,:\,(x-1)^2+y^2=1\}
\]
is irreducible, being a connected complex submanifold of $\mathbb{C}^2$. Therefore, also 
\[
\mathcal{A}=\{(x,y,z)\in\mathbb{C}^3\,:\,(x,y)\in\mathcal{A}',\ z^2=4-2x\}
\]
is irreducible, since $z^2-4+2x$ is irreducible.
Note also that the set 
\[
\mathcal{B}':=\{(u,v)\in\mathbb{C}^2\,:\,v=2u-u^2\}
\]
is locally irreducible in $\mathbb{C}^2$ as the graph of a quadratic function.
Consequently, $\mathcal{B}$ is also locally irreducible, since it is the "lifting"
of $\mathcal{B}'$ to the plane $\{(u,v,w)\in\mathbb{C}^3\,:\,w=4-2u\}$, namely
\[
\mathcal{B}=\{(u,v,w)\in\mathbb{C}^3\,:\,(u,v)\in\mathcal{B}',\ w=4-2u\}.
\]
Now, by Theorem \ref{VK_z_gory_z_dolu}, we have 
\[ V_{K,\B}\circ f = 4 \, V_{f^{-1}(K)} = 4 \, V_{[0,2]}\ \ \ \mbox{on} \ \ \ \CC. \]
On the other hand, the mapping
\[ g: \A \ni (x,y,z) \longmapsto (x,y^2,z^2) \in \B\]
is surjective, proper and $E=g^{-1}(K)$. Additionally, we can easily check that \ $\ro\!\!(g\!,\A)=\mathcal{L}_\infty(g|\A)=2$ \ and \ $g(x,y,z)=f(x)$ \ for all $(x,y,z)\in \A$. Once again we use Theorem \ref{VK_z_gory_z_dolu} and for any $(x,y,z)\in \A$ we get
\[ V_{E}(x,y,z) = V_{g^{-1}(K),\A}(x,y,z) = \tfrac12\, V_{K,\B} (g(x,y,z)) = \tfrac12\, V_{K,\B} (f(x))\] \[= 2\, V_{[0,2]}(x) = 2\, \log\left|x-1+\sqrt{x^2-2x}\right|, \]
the last equality being a consequence of the classical formula for the pluricomplex Green's function for $[-1,1]$. Since $\A$ and $\B$ do not have singular points, the same formulas are true for the functions $U_{g^{-1}(K),\A}$ and $U_{K,\B}$.

It is worth noticing here that every polynomial $p\in \CC[x,y,z]$ restricted to $\A$ can be written in the form \ $p(x,y,z) = p_1(x) + p_2(x)\,y + p_3(x)\, z+ p_4(x)\,yz$ \ where $p_1,\,p_2,\,p_3,\, p_4$ are polynomials of one complex variable $x$. 
Taking into account the Siciak-Zaharjuta result (\ref{Siciak-Za}), for any non-constant polynomial $p$ in the above form we have the following Bernstein-Walsh inequality on $\A$ for the Viviani's window $E$
\[ |p(x,y,z)| \le \left|x-1+\sqrt{x^2-2x}\right| ^{2\ro(p,\A)} \|p\|_E \ \ \ \ \mbox{for} \ \ \ (x,y,z)\in \A\setminus E\]
and the estimate is asymptotically optimal as $\ro\!\!(p,\A)$ \ tends to infinity.
\end{example}

\vskip 2mm

\noindent {\bf Acknowledgement.} We wish to thank Maciej Denkowski for useful discussions and for answering in \cite{BDT} a question posed by the first author during a conference presentation in Łódź, January 2020.

\end{document}